\documentclass[preprint,12pt]{elsarticle}

\usepackage{amsmath,amssymb,amsthm,mathtools}
\usepackage{booktabs}
\usepackage{array}
\usepackage{longtable}
\usepackage{float}
\usepackage{placeins}
\usepackage{graphicx}
\usepackage{microtype}
\usepackage{needspace}
\usepackage{xurl}
\usepackage[hidelinks]{hyperref}
\hypersetup{
  pdftitle={Finite-Query Coherent Projection of Black-Box Conditional Predictions},
  pdfauthor={Yuyang Wu, Fanyu Zhou, Jinfeng Zhou, Guangying Lyu}
}

\makeatletter
\patchcmd{\pprintMaketitle}{\vskip36pt}{\vskip12pt}{}{
  \PackageError{manuscript}{Unable to compact front-matter spacing}{}
}
\makeatother

\graphicspath{{figures/}}

\newtheorem{theorem}{Theorem}
\newtheorem{proposition}[theorem]{Proposition}
\newtheorem{lemma}[theorem]{Lemma}
\theoremstyle{remark}
\newtheorem{remark}[theorem]{Remark}

\begin{document}
\begin{frontmatter}

\title{Finite-Query Coherent Projection of Black-Box Conditional Predictions}
\author[nuist]{Yuyang Wu}
\author[nuist]{Fanyu Zhou}
\author[suda]{Jinfeng Zhou}
\author[nuist]{Guangying Lyu\corref{cor1}}
\ead{gylvmaths@nuist.edu.cn}
\cortext[cor1]{Corresponding author}
\address[nuist]{School of Mathematics and Statistics, Nanjing University of
Information Science and Technology, Nanjing, China}
\address[suda]{School of Computer Science and Technology, Soochow University,
Suzhou, China}

\begin{abstract}
Local conditional predictions from a black box need not admit a common
joint distribution.  We study two aspects of their coherent projection:
finite-access precision and directional fidelity.  For arbitrary
deterministic binary conditional tables, we establish matching expected
minimax excess conditional-Kullback--Leibler risk
$\Theta(\min\{R\sqrt{d/B},d^2/B\})$ under independent uniform exact-table
queries, with repeated draws charged, for a local bounded pairwise joint
class.  The result holds for sufficiently large $d$, $R\sqrt d\le M_0$, and
$Cd^3\le B\le(d2^{d-1})^{1/3}$, where $M_0$ is fixed.
A moment estimator attains the upper bound.  The main technical
contribution is a matching lower-bound construction that respects shared
pairwise interactions and frozen exact-table responses, with comparison
to the best projection of the same realized table.
We then examine directional fidelity in a computable hard-marginal
formulation that preserves direct black-box marginals.  A preregistered
TabPFN-v2 case study combines a Car anchor with a frozen two-dataset,
three-split extension.
All $1{,}152$ extension optimization instances pass independent numerical
checks at registered tolerances.  Directional contrasts are dataset-dependent
and asymmetric: optimizing one direction need not improve the other.
Together, the analyses distinguish accurate approximation of a selected
projection from improvement in conditional directions outside its fitted
objective.
\end{abstract}

\begin{keyword}
probabilistic coherence \sep joint reconstruction \sep query complexity
\sep black-box conditional prediction
\end{keyword}

\end{frontmatter}

\section{Introduction}
\label{sec:introduction}

Black-box conditional predictors provide local probabilistic answers that
may not agree with any single joint distribution.  For example, one binary
conditional table can make $X_1$ independent of $X_2$, while another makes
the distribution of $X_2$ depend on $X_1$; these cannot describe the same
positive joint law.  A coherent joint supplies a common source for
conditional and event probabilities used in combined reasoning.

Compatibility, joint construction, and compromises between incompatible
conditionals have a substantial statistical literature
\cite{arnold2001conditionally,arnold2002exact,kuo2017exactly,
arnold2022conditional}.  Probabilistic knowledge-base repair also studies
consistency-restoring adjustments
\cite{picadomuino2011repairing,potyka2017inconsistency}.
We build on these questions rather than propose joint construction or
consistency repair as new objectives.

Coherent reconciliation raises two questions: how accurately a specified
projection can be estimated from limited access, and which conditional
directions that projection fits well.  A learner observes black-box
probability answers and returns a joint in a chosen student class.  Its
conditional-Kullback--Leibler (KL) risk is evaluated against the same
teacher, under a specified distribution over conditional queries.
Figure~\ref{fig:reconciliation-pipeline} shows this projection task.
Comparing with the best member of the student class separates finite-access
error from disagreement that remains even at the projection optimum
(Section~\ref{sec:problem}).

\begin{figure}[htbp]
  \centering
  \includegraphics[width=0.82\linewidth]{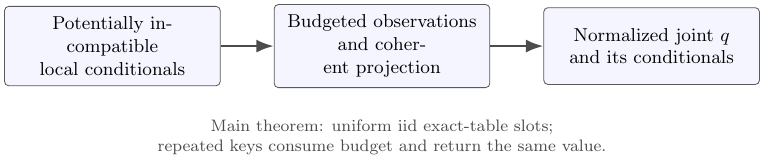}
  \caption{Finite-query coherent projection. Local conditional predictions
  are approximated by conditionals of a single joint. The sampling rule
  beneath the pipeline applies to the main theorem. The diagram was typeset in TikZ with assistance from OpenAI Codex (OpenAI),
  using GPT-5.6 Sol (model identifier: \texttt{gpt-5.6-sol}), and was reviewed
  and edited by the authors.}
  \label{fig:reconciliation-pipeline}
\end{figure}

Exact probability answers remove label-sampling noise at observed keys,
but leave unobserved entries of an arbitrary table unknown.  The first
question is therefore whether, and at what risk, a limited transcript
determines the best coherent projection.  We study this question for a
local pairwise joint class, whose undirected edge parameters are shared
across conditional directions.  A matching lower bound must respect this
design and compare with the optimum of the same realized teacher table.

Even an accurately computed projection need not fit every conditional
direction better: its objective determines which discrepancies are
penalized.  We examine this second question in a computable small-joint
setting that preserves direct black-box marginals.  Frozen model
predictions allow us to assess whether fitting one conditional direction
also improves the other.  The case study thus complements the finite-access
analysis by examining directional fidelity under a chosen projection
objective; its output class and observation protocol are different.

\paragraph{Contributions}
\begin{enumerate}
 \item Finite-access characterization.  We establish matching upper and
 lower bounds for the expected excess conditional-KL risk of projecting
 arbitrary deterministic binary conditional tables onto a local bounded
 pairwise joint class.  Under uniform passive exact-table queries, the
 result characterizes the minimax dependence on budget, dimension, and
 student radius over the stated window.  The main technical contribution
 is a matching lower-bound construction that jointly respects the legal
 shared-edge design, frozen exact-table responses, and comparison with the
 best projection of the same realized table.
 \item Directional trade-offs under fixed marginals.  In a preregistered
 frozen TabPFN-v2 study, we examine whether fitting one conditional
 direction improves the opposite direction while preserving direct
 marginals.  A confirmatory Car anchor and six external dataset--split
 units provide independent numerical checks and reveal dataset-dependent,
 asymmetric transfer: improvement in the fitted objective does not ensure
 improvement in the omitted direction.
\end{enumerate}

Supporting coverage, rank, and fixed-design results clarify these claims.
An independent restricted response-adaptive conditional-table task is
retained in Appendices~\ref{app:auxiliary-query-design}--\ref{app:supplementary-studies}.
Its output is not a joint, and its real-structure weighting study includes a
dataset-level reversal; it is not evidence for active coherent reconstruction.

Section~\ref{sec:related-work} compares the nearest compatibility and
statistical estimation results.  Section~\ref{sec:problem} defines the
projection risk and observation protocols.  Section~\ref{sec:theory}
states the full passive characterization and its limits.
Section~\ref{sec:method} gives the hard-marginal formulation, and
Section~\ref{sec:experiments} reports the frozen case study.
Section~\ref{sec:limitations} discusses the scope before the conclusion.
\section{Related Work}
\label{sec:related-work}

Two bodies of work are especially close: constructing coherent compromises
from conditionals, and estimating a best-in-class model with limited
information.
Table~\ref{tab:prior-work} compares selected results by observation,
output, and target; a broader comparison is retained in
Appendix~\ref{app:supplementary-studies}.

\paragraph{Compatibility and probabilistic repair}
Classical conditional specification studies existence, uniqueness, and
construction of joint laws~\cite{arnold2001conditionally,arnold2022conditional}.
Arnold et al.~\cite{arnold2002exact} formulate finite discrete exact,
near, and $\epsilon$-compatibility, while
Chen~\cite{chen2010compatibility} uses odds-ratio representations.
Kuo et al.~\cite{kuo2017exactly} introduce structural ratio matrices for
compatible and almost-compatible high-dimensional discrete conditional
models.  These are direct precedents for joint construction, including the
incompatible case; high dimension alone does not distinguish our objective.

Compromise criteria also predate this work.  Gibbs ensembles
\cite{chen2011gibbs} and random-scan pseudo-Gibbs compromises
\cite{mure2019optimal} treat inconsistent conditional specifications.
Dependency networks~\cite{heckerman2000dependency} allow local inconsistency,
and subset-conditional consistency is studied by
Majid et al.~\cite{majid2025consistent}.  Picado Mui\~no
\cite{picadomuino2011repairing}, Thimm~\cite{thimm2013inconsistency}, and
Potyka and Thimm~\cite{potyka2014consolidation,potyka2017inconsistency}
develop probabilistic inconsistency measurement, consolidation, and
reasoning with generalized models.  Divergence-based recovery and merging
also appear in probabilistic knowledge-based systems
\cite{nguyen2021probabilistic}.

Conditional-to-joint construction and conditional-closeness objectives also
occur for masked language models~\cite{torroba2023deriving}.  Conditional
probability queries can recover low-rank Bayesian-network structure under
appropriate compatibility assumptions~\cite{barik2019learning}.
Tractable probabilistic-circuit distillation~\cite{liu2023understanding}
addresses another student capability, while coordinate-conditional sampling
can support identity testing~\cite{blanca2023complexity}.  These query units,
outputs, and assumptions are not interchangeable with an exact frozen-table
projection risk.

\paragraph{Finite information and best-in-class estimation}
Derezi\'nski et al.~\cite{derezinski2019minimax} study limited-response
experimental design and least-squares estimation, including arbitrary fixed
responses.  Their selectable designs, squared-loss objective, and estimator
criteria differ from our uniform passive legal rows and bounded proper
logistic output.

Mourtada and Ga\"iffas~\cite{mourtada2022improper} give an improper estimator
and excess-risk bounds for misspecified density estimation and logistic
regression.  Their hard-label observation model and potentially out-of-class
predictor differ from our exact probability responses and required coherent
student.  Loss similarities do not make the experiments equivalent:
probability responses can generate labels, so applicable label-based upper
bounds may transfer, but a hard-label lower bound does not automatically
lower-bound the more informative probability oracle.

Moment and curvature arguments give our upper bound; stochastic averaging
offers another route~\cite{lacostejulien2012simpler}.  Matching concerns
the shared-edge design, frozen observations, proper bounded output, and
realized-table comparator.

\paragraph{Constrained fitting and the TFM case study}
Information projection~\cite{csiszar1975idivergence}, Bregman methods
\cite{bregman1967relaxation,benamou2015iterative}, and constrained marginal
fitting~\cite{evans2013two} provide precedents for our hard-marginal
formulation.  Section~\ref{sec:method} specifies its forward conditional
KL, cell weights, lower bounds, and numerical checks.

TabPFN~\cite{hollmann2025tabpfn} supplies the frozen case-study predictions.
TFM consistency diagnostics~\cite{klotergens2026agree} motivate the question,
while the heterogeneity of tabular learning
\cite{shwartzziv2022tabular,borisov2024deep} cautions against generalizing
from one checkpoint.

\begingroup
\setlength{\intextsep}{3pt}
\setlength{\abovecaptionskip}{3pt}
\begin{table}[!htb]
 \centering
 \footnotesize
 \linespread{0.95}\selectfont
 \caption{Closest joint-construction and estimation results in their stated settings.}
 \label{tab:prior-work}
 \begin{tabular}{@{}>{\raggedright\arraybackslash}p{0.23\linewidth}
                    >{\raggedright\arraybackslash}p{0.33\linewidth}
                    >{\raggedright\arraybackslash}p{0.36\linewidth}@{}}
 \toprule
 Work & Observation and target & Relevant distinction \\
 \midrule
 Arnold et al.; Kuo et al.~\cite{arnold2002exact,kuo2017exactly} &
 Discrete conditional specifications; compatible and near-compatible joints &
 Joint compatibility and compromise, versus finite-transcript excess risk. \\
 \addlinespace[1pt]
 Picado Mui\~no; Potyka--Thimm~\cite{picadomuino2011repairing,potyka2017inconsistency} &
 Explicit probabilistic statements; repair and generalized models &
 Repair and generalized reasoning, versus partial-access estimation of
 the same-table projection. \\
 \addlinespace[1pt]
 Torroba Hennigen--Kim~\cite{torroba2023deriving} &
 Masked-LM conditionals; joint construction by conditional closeness &
 Conditional-closeness objective; here the student class and access
 protocol are fixed. \\
 \addlinespace[1pt]
 Derezi\'nski et al.~\cite{derezinski2019minimax} &
 Hidden fixed responses; selectable least-squares design &
 Squared loss and selectable design, versus shared-edge logistic
 projection. \\
 \addlinespace[1pt]
 Mourtada--Ga\"iffas~\cite{mourtada2022improper} &
 Labels; misspecified log loss with improper output &
 Hard labels and improper output, versus exact probabilities and bounded
 proper output. \\
 \addlinespace[1pt]
 This work &
 Arbitrary binary tables; uniform iid exact responses; pairwise joint output &
 Matching same-table projection excess risk within the stated radius and
 budget window. \\
 \bottomrule
 \end{tabular}
\end{table}
\FloatBarrier
\endgroup
\section{Problem Formulation}
\label{sec:problem}

Let $X=(X_1,\ldots,X_d)$ be finite categorical variables, with
$X_j\in\mathcal X_j$ and product alphabet
$\mathcal X:=\prod_{j=1}^d\mathcal X_j$.  Write $x\in\mathcal X$ for a
complete realization.  For an index set
$T\subseteq[d]:=\{1,\ldots,d\}$, $x_T$ denotes the coordinates indexed by
$T$, and $-T:=[d]\setminus T$ (in particular, $-j=[d]\setminus\{j\}$).
A black-box oracle $\mathcal O$ maps a query key $z=(j,S,\bar x_S)$, with
$S\subseteq-j$, to a categorical probability vector
$\widetilde p_j(\cdot\mid \bar x_S)$ on $\mathcal X_j$, where
$\bar x_S\in\prod_{k\in S}\mathcal X_k$ is the fixed conditioning
assignment.  The responses need not
arise from a common joint distribution.  We write
$\sigma(t)=(1+\exp(-t))^{-1}$ for the logistic link used below, and
$\operatorname{Ber}(\pi)$ for the Bernoulli distribution on $\{0,1\}$ with
success probability $\pi$.  We also write
\[
  \Delta_+(\mathcal X)
  :=\left\{q:\mathcal X\to(0,1):
  \sum_{x\in\mathcal X}q(x)=1\right\}
\]
for the strictly positive probability simplex.  Each
$q\in\Delta_+(\mathcal X)$ induces conditionals by marginalization:
\begin{equation}
  q_j(a\mid \bar x_S)
  =
  \frac{\displaystyle
    \sum_{\substack{x\in\mathcal X:\;x_j=a,\;x_S=\bar x_S}}q(x)}
       {\displaystyle
    \sum_{\substack{x\in\mathcal X:\;x_S=\bar x_S}}q(x)},
  \qquad a\in\mathcal X_j .
  \label{eq:induced-conditional}
\end{equation}
Both sums range over complete configurations; positivity makes the
denominator nonzero.  Global coherence means that every reported student
conditional is induced by this same normalized joint.

For probability vectors $p$ and $r$ on the same finite alphabet, define the
forward Kullback--Leibler divergence~\cite{kullback1951information} by
\[
  D_{\mathrm{KL}}(p\|r)
  :=\sum_a p(a)\log\frac{p(a)}{r(a)},
\]
with $0\log(0/r):=0$.  All student probabilities below are positive, so no
division by zero occurs.

For a fixed evaluation distribution $\mu$ over query keys, we measure
teacher-to-student conditional disagreement by
\begin{equation}
  R_\mu(q;\mathcal O)
  =\mathbb E_{(j,S,\bar x_S)\sim\mu}
  D_{\mathrm{KL}}\!\left(
    \widetilde p_j(\cdot\mid \bar x_S)\,\middle\|\,
    q_j(\cdot\mid \bar x_S)
  \right).
  \label{eq:conditional-kl-risk}
\end{equation}
Fixing $\mu$ is part of the task definition: a best coherent joint is not
well-defined without an evaluation distribution.  For a coherent student class
$\mathcal Q\subseteq\Delta_+(\mathcal X)$ and a budget-$B$ output
$\widehat q_B\in\mathcal Q$, introduce the two population risk floors
\[
  R_\mu^\star(\mathcal O)
  :=\inf_{q\in\Delta_+(\mathcal X)}R_\mu(q;\mathcal O),
  \qquad
  R_{\mu,\mathcal Q}^\star(\mathcal O)
  :=\inf_{q\in\mathcal Q}R_\mu(q;\mathcal O).
\]
We then define
\begin{align}
  I_\mu(\mathcal O)
  &:=R_\mu^\star(\mathcal O),
  \label{eq:irreducible-risk}\\
  A_{\mu,\mathcal Q}(\mathcal O)
  &:=R_{\mu,\mathcal Q}^\star(\mathcal O)
      -R_\mu^\star(\mathcal O),
  \label{eq:approximation-risk}\\
  E_{B,\mu,\mathcal Q}(\widehat q_B;\mathcal O)
  &:=R_\mu(\widehat q_B;\mathcal O)
     -R_{\mu,\mathcal Q}^\star(\mathcal O).
  \label{eq:excess-risk}
\end{align}
These definitions give the accounting identity
\begin{equation}
  R_\mu(\widehat q_B;\mathcal O)
  =I_\mu(\mathcal O)+A_{\mu,\mathcal Q}(\mathcal O)
   +E_{B,\mu,\mathcal Q}(\widehat q_B;\mathcal O).
  \label{eq:risk-decomposition}
\end{equation}
This telescoping identity is not an additional statistical guarantee.
The terms are nonnegative: irreducible disagreement, class approximation,
and finite-query or optimization excess, respectively.  Only the last
depends on the learner's transcript and can be random; we abbreviate it as
$E_{B,\mathcal Q}$ when the other arguments are fixed, and take expectations
after the pointwise decomposition.

Crucially, $I_\mu=0$ concerns only positive-$\mu$ keys; global oracle
compatibility additionally requires sufficient coverage and, when relevant,
attainment by a positive joint.  Also, $A_{\mu,\mathcal Q}$ includes all
student restrictions, such as hard marginals, not just pairwise structure.
The restricted optimum must not be relabelled as pure incompatibility.

\paragraph{Notation used in the theory}
For vectors $u,v\in\mathbb R^m$, $\langle u,v\rangle=u^\top v$ and
$\|u\|_2=(\sum_{r=1}^m u_r^2)^{1/2}$.  The indicator of an event $E$ is
$\mathbf 1_{\{E\}}$.  Vector inequalities are componentwise unless a
matrix order is explicitly stated.

\paragraph{Query protocols and budget semantics}

The oracle model and budget convention are part of every claim.  The main
theory uses exact deterministic conditional-probability tables.  Its
passive-i.i.d. protocol samples query keys independently from a stated design
distribution; repeated occurrences of a key consume design slots but return
the same frozen table entry.  They are not independent noisy observations and
do not entitle the learner to draw a replacement key.  Quantized and noisy
response models have distinct information budgets and are not silently
identified with this exact-table protocol.

An adaptive policy chooses keys using its transcript, whereas a static
policy fixes its allocation before observing responses.  The main matching
theorem uses neither adaptive selection nor a guaranteed-distinct budget.
The supplementary adaptive construction
(Appendix~\ref{app:auxiliary-query-design}) instead reports a non-joint
conditional table and has a different loss.  Its guarantee does not
transfer to coherent reconstruction.

Table~\ref{tab:claim-evidence-bridge} identifies the complementary roles of
finite-access precision and directional fidelity, together with their
distinct protocols and the supplementary studies.

\begin{table}[!htbp]
 \centering
 \small
 \caption{Finite-access precision and directional fidelity: complementary
 projection tasks with distinct protocols and validation targets.}
 \label{tab:claim-evidence-bridge}
 \begin{tabular}{@{}>{\raggedright\arraybackslash}p{0.20\linewidth}
                    >{\raggedright\arraybackslash}p{0.73\linewidth}@{}}
 \toprule
 Task & Observation, target, and evidence \\
 \midrule
 Passive theory &
 Arbitrary deterministic binary full-conditional table; uniform iid exact
 slots, repeats charged.  Same-table best bounded pairwise projection;
 matching expected excess conditional KL within
 Theorem~\ref{thm:passive-minimax}'s radius and budget window. \\
 \addlinespace
 TFM case study &
 Frozen direct marginals and both conditional families; small hard-marginal
 joint tables.  Independent numerical checks and within-unit directional
 risk contrasts, not passive-rate or predictive-accuracy validation. \\
 \addlinespace
 Supplementary tasks &
 Quantized fixed-pool recovery and noisy non-joint table prediction have
 separate losses and budgets
 (Appendices~\ref{app:auxiliary-query-design}--\ref{app:supplementary-studies}).
 Their outcomes are not pooled with either main task. \\
 \bottomrule
 \end{tabular}
\end{table}
\FloatBarrier
\section{Passive Finite-Query Characterization}
\label{sec:theory}

How accurately can exact local probability answers determine a coherent
projection when only a fraction of the conditional table is observed?
The main result gives matching upper and lower bounds for arbitrary
deterministic binary teacher tables and a local bounded pairwise student
class.  Shared undirected parameters enforce coherence of the output;
the teacher need not be compatible with any joint.

\subsection{Coherent pairwise output}

For $X\in\{-1,+1\}^d$, set $m:=d+d(d-1)/2$ and write
\[
  \theta
  :=\bigl(b_1,\ldots,b_d,(w_{jk})_{1\le j<k\le d}\bigr)^\top
  \in\mathbb R^m,
\]
The edge block is ordered lexicographically by $(j,k)$, and $w_{jk}=w_{kj}$
when referenced in either order.  The positive pairwise joint is
\begin{equation}
  q_\theta(x)
  :=\frac{1}{\mathcal N(\theta)}
  \exp\!\left\{
    \frac12\sum_{j=1}^d b_jx_j
    +\frac12\sum_{1\le j<k\le d}w_{jk}x_jx_k
  \right\},
  \label{eq:pairwise-student}
\end{equation}
where $\mathcal N(\theta)$ is the finite normalizing constant.  At a full-conditional
key $z=(j,x_{-j})$, define $a(z)\in\mathbb R^m$ to have entry $1$ at the
$j$th intercept, entry $x_k$ at each edge $\{j,k\}$ incident to $j$, and zero
elsewhere.  Then
\begin{equation}
  q_\theta(X_j=+1\mid X_{-j}=x_{-j})
  =\sigma\!\left(b_j+\sum_{k\ne j}w_{jk}x_k\right)
  =\sigma(a(z)^\top\theta).
  \label{eq:pairwise-logit}
\end{equation}
Thus one shared edge parameter is used in both conditional directions.
Write $\Theta_R:=\{\theta\in\mathbb R^m:\|\theta\|_2\le R\}$ and
$\mathcal Q_R:=\{q_\theta:\theta\in\Theta_R\}$.

\subsection{Uniform passive exact-table protocol}
\label{sec:passive-protocol}

For clarity, let
$\mathcal Z:=\{(j,x_{-j}):j\in[d],
x_{-j}\in\{-1,+1\}^{d-1}\}$ and let $\mathcal P_{\rm all}$ be the set of
all deterministic tables
$p:\mathcal Z\to[0,1]$ on the $N=d2^{d-1}$ full-conditional keys.
Values zero and one are allowed; the student conditionals remain strictly
positive, so the KL risk is finite under the convention in
Section~\ref{sec:problem}.
Throughout this subsection, $\mu$ is uniform on $\mathcal Z$.
Under $B$ independent uniform draws from
$\mathcal Z$, let $\mathcal O_p$ denote the oracle returning $p(z)$ and define
\begin{equation}
  \mathfrak M_{B,\mathrm{iid}}^{\mathrm{agn}}(d,R)
  :=\inf_{\mathcal A}\sup_{p\in\mathcal P_{\rm all}}
  \mathbb E\!\left[
    R_\mu(q_{\widehat\theta};\mathcal O_p)
    -\inf_{\theta\in\Theta_R}R_\mu(q_\theta;\mathcal O_p)
  \right],
  \label{eq:main-minimax-definition}
\end{equation}
where $\mathcal A$ ranges over randomized learners whose output lies in
$\Theta_R$, and the expectation is over the passive design and learner
randomization.  
Set $\kappa_0:=\sigma(M_0)(1-\sigma(M_0))$ for a fixed $M_0>0$.
The following statement includes the radius and budget restrictions;
Appendix~\ref{sec:minimax-theory} supplies its complete proof.

The upper bound holds uniformly over $\mathcal P_{\rm all}$
(Proposition~\ref{prop:moment-upper}).  For the lower bound, it suffices to
use the subclass with entries in $\{1/4,3/4\}$, obtained by fixing
$\gamma=1/4$ in Appendix~\ref{sec:minimax-theory}'s frozen-table
construction.  Inclusion of this subclass in $\mathcal P_{\rm all}$ gives
the same lower bound for the full teacher class.  Thus the theorem does
not require teacher probabilities to be close to $1/2$; the local radius
restriction concerns the student.

\Needspace{16\baselineskip}
\begin{theorem}[Passive exact-table minimax rate]
\label{thm:passive-minimax}
Fix $M_0>0$.  There exist constants
$c,C,d_0>0$, depending only on $M_0$, such that if
\[
  d\ge d_0,
  \qquad B\in\mathbb N,
  \qquad 0<R,
  \qquad R\sqrt d\le M_0,
  \qquad Cd^3\le B\le(d2^{d-1})^{1/3},
\]
then
\begin{equation}
  c\min\left\{
    R\sqrt{\frac dB},\frac{d^2}{B}
  \right\}
  \le
  \mathfrak M_{B,\mathrm{iid}}^{\mathrm{agn}}(d,R)
  \le
  \min\left\{
    2R\sqrt{\frac dB},
    \frac{d^2}{2\kappa_0B}
  \right\}.
  \label{eq:matching-rate}
\end{equation}
Consequently, over this window,
\[
  \mathfrak M_{B,\mathrm{iid}}^{\mathrm{agn}}(d,R)
  =\Theta_{M_0}
  \left(\min\{R\sqrt{d/B},d^2/B\}\right).
\]
\end{theorem}

Each repeated key consumes a slot and returns the same frozen response.
The learner cannot choose the key distribution or replace a repeated draw.
The comparison is with the best member of $\Theta_R$ for the same realized
table, not with a joint presumed to generate that table.

\subsection{Why exact answers leave finite-access uncertainty}

Exact responses remove label-sampling noise at an observed key, but do not
reveal unobserved entries of an arbitrary table.  The projection depends
on a moment over all keys, so a limited passive transcript can leave its
optimum uncertain.  The lower bound quantifies this uncertainty within the
shared-edge design, rather than importing a hard-label lower bound into
a more informative observation model.

Write $A(t)=\log(1+e^t)$,
$G(\theta)=\mathbb E_\mu A(a(Z)^\top\theta)$, and
$b_p=\mathbb E_\mu[p(Z)a(Z)]$.  Up to a teacher-dependent constant, the
population risk is $G(\theta)-b_p^\top\theta$.  Replacing $b_p$ by its
sample average yields the moment estimator in
Eq.~\eqref{eq:moment-estimator}.  Its upper bound follows from a
second-moment estimate and logistic curvature
(Proposition~\ref{prop:moment-upper}).  These are standard estimation tools:
projected stochastic subgradient averaging also gives the familiar
strongly convex rate~\cite{lacostejulien2012simpler}.  The contribution is
the matching characterization for the specified projection experiment,
rather than a new general optimizer.

The lower bound must additionally respect the fixed legal feature rows and
the exact-response protocol.  Its edge-code construction generates a
deterministic table before observation; querying a repeated key reveals no
new sign.  Concentration of the realized table moment then connects code
separation to the optimum of that actual table.  Interior and boundary
arguments, with a transition step, give the two regimes.  The proof uses
standard information lower-bound tools~\cite{yao1977probabilistic,
tsybakov2009introduction}; its problem-specific work is the legal frozen-table
construction and the realized-table comparison.

The estimator treats the population design term $G$ as known.  Evaluating
its exponential sum is not asserted to take polynomial time.  The theorem
is therefore an information result, not computational minimax optimality.

\subsection{Rate interpretation and student-class scope}

Let $S=R\sqrt d$ and $m=d(d+1)/2$.  Since $m$ is of order $d^2$, the rate
can equivalently be written, up to constants, as
$\min\{S/\sqrt B,m/B\}$.  The two displayed branches meet at
$B=d^4/S^2$.  These are familiar norm- and dimension-dependent scales;
the theorem identifies their matching role for the present observation
and output constraints.  For a given radius, the crossover need not lie
inside the stated budget window.

The radius restriction has a substantive modeling effect.  The uniform
design covariance in Lemma~\ref{lem:covariance} implies
\begin{equation}
 \mathbb E_\mu[(a(Z)^\top\theta)^2]
 \le \frac{2R^2}{d}
 \le \frac{2M_0^2}{d^2},
 \qquad \theta\in\Theta_R.
 \label{eq:student-signal-scope}
\end{equation}
Thus the root-mean-square conditional logit decreases with $d$ under this
evaluation design, although the worst-key logit can remain bounded away
from zero.  A single fixed nonzero edge weight eventually violates the
radius condition, even though its conditional logits remain uniformly
bounded.  The theorem does not cover general fixed-strength sparse
interactions.  This is a local restriction on the student, not on the
arbitrary teacher table.  Vanishing excess risk approaches the selected
projection optimum; it need not eliminate the total teacher--student
disagreement.  The matching guarantee covers the stated budget window,
whose lower endpoint is of order $d^3$, rather than all query budgets.

\subsection{Supporting comparisons}

Two simpler comparisons clarify why observation protocol and comparator
must be kept separate.  The first is an unstructured coverage barrier;
the second is exact identification of an already compatible teacher.
Neither comparison measures the price of incompatibility by subtracting
risks from different protocols.

\subsubsection{Unstructured coverage}
\label{sec:unstructured-barrier}

This comparison illustrates the coverage cost of an unstructured family.
Let $X=(Y,Z_1,\ldots,Z_{d-1})\in\{0,1\}^d$, query only the full conditional of
$Y$, and evaluate risk uniformly over the $N_{\rm ctx}=2^{d-1}$ contexts of
$Z$.  For each sign vector $\omega\in\{-1,+1\}^{N_{\rm ctx}}$ and
$0<\gamma<1/2$, consider the
compatible teacher family
\begin{equation}
  \Pr_\omega(Y=1\mid Z=z)=\frac12+\gamma\omega_z.
  \label{eq:unstructured-hard-family}
\end{equation}
Each member is induced by a valid joint distribution, so the irreducible
incompatibility term is zero in this subproblem.

\begin{proposition}[Unstructured coverage barrier]
\label{thm:unstructured-barrier}
For the exact full-conditional oracle in
Eq.~\eqref{eq:unstructured-hard-family}, every randomized, possibly adaptive
algorithm using at most $B$ queries satisfies
\begin{equation}
  \sup_\omega\mathbb E R_\mu(\widehat q_B;\mathcal O_\omega)
  \ge
  \left(1-\frac{B}{N_{\rm ctx}}\right)_+\kappa(\gamma)
  \ge
  2\gamma^2\left(1-\frac{B}{2^{d-1}}\right)_+,
  \label{eq:unstructured-barrier}
\end{equation}
where
\(
\kappa(\gamma)=D_{\mathrm{KL}}(
\operatorname{Ber}(1/2+\gamma)\|\operatorname{Ber}(1/2))
\) and $x_+:=\max\{x,0\}$.
In particular, risk at most $\gamma^2$ requires $B\ge2^{d-2}$.
\end{proposition}

This is a standard independent-sign coverage argument using Yao's
principle~\cite{yao1977probabilistic}, not a universal oracle lower bound.
Appendix~\ref{app:unstructured-proof} checks compatibility and the
unqueried-context Bayes loss for this concrete family.

\subsubsection{Compatible exact-response rank boundary}
\label{sec:rank-boundary}

For a compatible teacher $p_\theta(z)=\sigma(a(z)^\top\theta)$, taking an
exact response's logit supplies one linear measurement of $m$ parameters.
This equation-solving mechanism is established in logistic model
extraction~\cite[Section~4.1.1]{tramer2016stealing}.  Here uniform full-key
evaluation and output in $\Theta_R$ give the adaptive risk lower bound
\begin{equation}
  \inf_{\mathcal A_{\rm ad}:\,B\text{ exact calls}}
  \sup_{\theta\in\Theta_R}
  \mathbb E R_{\mu_{\rm unif}}
  (q_{\widehat\theta};\mathcal O_\theta)
  \ge\frac{\kappa_0R^2}{2d}\frac{(m-B)_+}{m+2},
  \label{eq:compatible-rank}
\end{equation}
where $\kappa_0=\sigma(M_0)(1-\sigma(M_0))$ and
$M_0>0$ is a fixed logit envelope satisfying $R\sqrt d\le M_0$; hence
$0<\kappa_0\le1/4$.
Appendix~\ref{app:rank-proof} verifies the uniform-ball residual uncertainty,
the curvature conversion, and the spanning property of our legal rows.
Consequently $m$ linearly independent exact queries recover a compatible
teacher.  That identification boundary is distinct from passive projection
of an arbitrary frozen table; the cited model-extraction result is not
being used as a proof of Eq.~\eqref{eq:compatible-rank}.

A fixed-design high-probability upper bound and a residual approximation
bound are retained in Appendix~\ref{app:supplementary-studies}.  They are
supporting inequalities, not additional main contributions.
\Needspace{10\baselineskip}
\section{Hard-Marginal Reconciliation}
\label{sec:method}

Having characterized finite-access precision for one student class, we
turn to directional fidelity: what does a chosen coherent projection fit
well?  We use a computable small joint table whose direct marginals are
fixed to black-box predictions.  Varying the weight on the two conditional
directions lets us examine whether improvement in the fitted objective
transfers to the other direction.  This is a complementary instance of
the conditional-risk framework, with a different output class and
observation design; it is not an implementation or empirical test of
Theorem~\ref{thm:passive-minimax}.

Let the two categorical targets be $A\in\mathcal A$ and
$B\in\mathcal B$, with $n_A:=|\mathcal A|$ and
$n_B:=|\mathcal B|$.  For each test covariate vector $x$, four black-box
prediction tasks provide the two direct marginals and both conditional
directions.  Fix $x$ and write $q_{ab}:=q_x(A=a,B=b)$ and
$\mathbf q:=\operatorname{vec}((q_{ab})_{a,b})$.  Define the two uniformly
weighted directional risks
\begin{align}
 R_{A\mid B,x}(q)
 &:=\frac{1}{n_B}\sum_b
 D_{\mathrm{KL}}\!\left(p_{A\mid B}(\cdot\mid b,x)
 \middle\|q_{A\mid B}(\cdot\mid b,x)\right),\notag\\
 R_{B\mid A,x}(q)
 &:=\frac{1}{n_A}\sum_a
 D_{\mathrm{KL}}\!\left(p_{B\mid A}(\cdot\mid a,x)
 \middle\|q_{B\mid A}(\cdot\mid a,x)\right),
 \label{eq:tfm-directional-risks}
\end{align}
and, for $\omega\in[0,1]$,
\begin{equation}
 R_x^{(\omega)}(q)
 :=\omega R_{A\mid B,x}(q)+(1-\omega)R_{B\mid A,x}(q).
 \label{eq:tfm-weighted-risk}
\end{equation}
The confirmatory direction-balanced estimand is
\begin{equation}
 R_{\mu_x^{\rm TFM}}(q_x;\mathcal O_x):=R_x^{(1/2)}(q_x).
 \label{eq:tfm-main-risk-specialization}
\end{equation}
The extension additionally uses $\omega=1$ for an $A\mid B$-only objective
and $\omega=0$ for a $B\mid A$-only objective.  The opposite conditional
direction is not included in either endpoint objective.

The direct predictions $p_A(\cdot\mid x)$ and $p_B(\cdot\mid x)$ constrain
$q_x$ and are not extra risk terms.  On the fixed-marginal feasible set,
$q_{A\mid B}(a\mid b,x)=q_{ab}/p_B(b\mid x)$ and
$q_{B\mid A}(b\mid a,x)=q_{ab}/p_A(a\mid x)$.  Hence, up to a constant that
does not depend on $q$,
\begin{equation}
 \begin{aligned}
 R_x^{(\omega)}(q)
 &=K_x^{(\omega)}-\sum_{a,b}c_{ab}^{(\omega)}\log q_{ab},\\
 c_{ab}^{(\omega)}
 &=\frac{\omega}{n_B}p_{A\mid B}(a\mid b,x)
 +\frac{1-\omega}{n_A}p_{B\mid A}(b\mid a,x).
 \end{aligned}
 \label{eq:tfm-coefficients}
\end{equation}
All reported probabilities are floored at $10^{-9}$ and renormalized once
before risk evaluation.  Let $\mathsf M$ be the full-row-rank marginalization
matrix retaining every row equation and all but the last column equation, and
let $\mathbf r$ contain the corresponding direct marginals.  With
$\underline q_{ab}:=10^{-8}p_A(a\mid x)p_B(b\mid x)$, define the feasible set
by
\[
 \mathcal P_x
 :=\left\{\mathbf q:
 \mathsf M\mathbf q=\mathbf r,\ 
 \mathbf q\ge\underline{\mathbf q}\right\}.
\]
The direction-weighted objective and its optimizer are
\begin{equation}
 \begin{aligned}
 \mathcal J_x^{(\omega)}(\mathbf q)
 &:= -\sum_{a,b}c_{ab}^{(\omega)}\log q_{ab},\\
 \mathbf q_x^{\star,(\omega)}
 &:=\operatorname*{arg\,min}_{\mathbf q\in\mathcal P_x}
 \mathcal J_x^{(\omega)}(\mathbf q).
 \end{aligned}
 \label{eq:tfm-transport}
\end{equation}
The independent joint
$q_x^{\rm ind}(a,b)=p_A(a\mid x)p_B(b\mid x)$ is strictly feasible.

Compactness, strict convexity, and standard convex duality
\cite[Chapter~5]{boyd2004convex} give a unique optimum and a primal--dual
suboptimality bound.  The Frank--Wolfe gap is another standard certificate
\cite[Section~2]{jaggi2013revisiting}, valid for exactly feasible points
regardless of which optimizer produced them.  Proposition~\ref{prop:tfm-transport}
in Appendix~\ref{app:transport-proof} collects the cell response, both
certificates, and the checks specific to our constraints.  These are
applications of established optimality tools, not new general algorithms.
Since the independent joint is feasible,
\begin{equation}
 R_x^{(\omega)}(q_x^{\star,(\omega)})
 \le R_x^{(\omega)}(q_x^{\rm ind}).
 \label{eq:tfm-optimized-noninferiority}
\end{equation}
This ordering concerns only the optimized objective.  When the two frozen
conditional directions are incompatible, it gives no ordering for the
direction omitted at $\omega=0$ or $1$.  Thus the observed magnitude and
certificate coverage are empirical findings, whereas the sign of exact
optimized-objective non-inferiority is structural.

We solve Eq.~\eqref{eq:tfm-transport} in dual coordinates with a damped
semismooth Newton method, following the established nonsmooth Newton
framework~\cite{qi1993nonsmooth}.  Every returned floating-point solution is checked
independently at registered tolerances for positivity, normalization,
marginal feasibility, the lower bound, a stable Fenchel primal--dual gap, and
a projected-gradient Frank--Wolfe gap.  Solver termination alone is not
accepted as a certificate, and the checks remain numerical rather than proofs
that stored floating-point rows are exactly feasible.  No general convergence
rate for the Newton implementation is claimed.

The relation to classical probability fitting is discussed in
Section~\ref{sec:related-work}.  Here the variable-dependent objective is
$-\sum c_{ab}^{(\omega)}\log q_{ab}$, equivalently the $q$-dependent part of
$D_{\rm KL}(c^{(\omega)}\|q)$.  It is not a regularized transport objective
with the reverse KL direction.
\section{Frozen TFM Case Study}
\label{sec:experiments}
\label{sec:tfm-validation}

The case study examines directional fidelity under the hard-marginal
projection in Section~\ref{sec:method}: does optimizing one conditional
direction improve the direction omitted from its objective?  Independent
numerical checks assess the computed solutions across the registered units.
A TabPFN-v2~\cite{hollmann2025tabpfn}/OpenML Car
split~\cite{vanschoren2013openml} is the confirmatory anchor.  A separately
frozen extension uses two additional categorical datasets and three splits
per dataset.  Together, these layers assess the formulation's numerical
reliability and directional consequences, rather than estimate the passive
minimax rate.

Across both layers, all four prediction tasks use the same TabPFN-v2
checkpoint through \texttt{tabpfn} 8.1.0, with
\texttt{n\_estimators=1}, \texttt{low\_memory} fit mode, and
\texttt{float32} inference.  Probability vectors are frozen before
optimization; reconciliation neither refits nor updates the TFM.

\paragraph{AI-assisted research workflow}
OpenAI Codex (OpenAI), using GPT-5.6 Sol (model identifier:
\texttt{gpt-5.6-sol}), assisted in preparing and revising research scripts,
zero-model-output preflight checks, offline verifiers, and tests.  Formal runs
followed frozen protocols with author authorization.  Dataset and checkpoint
choices and protocol changes remained under author control.  Reported values
came from registered computational outputs checked by protocol-defined
verifiers, not generated textual answers.  The authors reviewed and edited
AI-assisted material, checked the results, and take responsibility for the
research decisions and conclusions.

\subsection{Confirmatory Car anchor}

The method was reached through three non-interchangeable registered stages.
T0 reduced conditional risk without preserving the directly predicted
marginals.  T1 imposed exact marginals, but a generic SLSQP implementation
passed the main numerical gates on only $8/64$ rows.  T2 changed the numerical
method, not the estimand or scientific thresholds, and used a new split that
excluded all T0/T1 test positions.  The full failure sequence and the T2
row-wise distribution are retained in Appendix~\ref{app:tfm-development-sequence}.

The T2 split used OpenML Car dataset 40975, version 3~\cite{dataset_car}, containing 1,728
complete rows.  Target $A=\texttt{class}$ has four classes, target
$B=\texttt{safety}$ has three classes, and the five covariates are
\texttt{buying}, \texttt{maint}, \texttt{doors}, \texttt{persons}, and
\texttt{lug\_boot}.  With random state 20260924, the disjoint confirmatory
split contains 1,024 training rows and 64 test rows.  All 64 transport
problems passed the independent numerical checks.  The mean
direction-balanced conditional risk decreased from $0.5040$ for the
independent joint to $0.03992$ for the anchored joint, an aggregate relative
reduction of $92.08\%$ against the preregistered $20\%$ threshold.  Direct
and reconciled marginal accuracies were identical, as expected because the
direct marginals were fixed.
Table~\ref{tab:tfm-confirmatory} summarizes the confirmatory results.

\begin{table}[htbp]
  \centering
  \small
  \caption{Confirmatory hard-marginal reconciliation on one new
  TabPFN-v2/Car split.  Both joints are assessed against the same four oracle
  predictions used in reconciliation; the independent baseline is the
  product of the direct marginals.  Certificate checks are independent and
  numerical, at registered tolerances.}
  \label{tab:tfm-confirmatory}
  \begin{tabular}{p{0.57\linewidth}r}
    \toprule
    Measure & Recorded value \\
    \midrule
    Independent-joint mean conditional risk & 0.5040015 \\
    Reconciled-joint mean conditional risk & 0.0399169 \\
    Aggregate relative conditional-risk reduction & $92.08\%$ \\
    Rows passing independent numerical checks & $64/64$ \\
    Maximum direct-marginal probability change & $9.09\times10^{-14}$ \\
    \bottomrule
  \end{tabular}
\end{table}

\subsection{Frozen external-validity extension}

The extension retains the same TabPFN-v2 checkpoint, hard-marginal solver,
lower bound, and certificate tolerances.  It uses OpenML Mushroom (ID 24,
version 1)~\cite{dataset_mushroom} and Bank Marketing (ID 1461,
version 1)~\cite{dataset_bank_marketing}, with three fixed splits
per dataset, 1,024 training rows and 64 test rows per split.  The three test
sets within each dataset are disjoint; training sets may overlap but exclude
all three registered test sets.  Dataset eligibility, target identities,
class order, and split positions were frozen by zero-model-output screening
before any extension model fit or prediction.  Four TabPFN tasks per unit
produce the two direct marginals and the two conditional families, for 24
fits and 384 test rows in total.

After the registered complete-case rule, Mushroom contains 5,644 rows;
target $A=\texttt{class}$ has two classes, target
$B=\texttt{cap-color}$ has eight classes, and the remaining 21 variables are
covariates.  Bank Marketing contains 45,211 rows; target
$A=\texttt{Class}$ has two classes, target $B=\texttt{V9}$ (contact
communication type) has three classes, and the remaining 15 variables are
covariates.

For a unit $u$ with test covariates $x_1,\ldots,x_{64}$, the bidirectional
gain and its aggregate relative reduction are
\begin{align}
 \widehat G_u^{\rm bi}
 &:=\frac{1}{64}\sum_{i=1}^{64}
 \left[R_{x_i}^{(1/2)}(q_{x_i}^{\rm ind})
      -R_{x_i}^{(1/2)}(q_{x_i}^{(1/2)})\right],\notag\\
 \widehat\rho_u^{\rm bi}
 &:=\frac{\widehat G_u^{\rm bi}}
 {\frac{1}{64}\sum_{i=1}^{64}R_{x_i}^{(1/2)}(q_{x_i}^{\rm ind})}.
 \label{eq:tfm-bidirectional-endpoints}
\end{align}
The two direction-held-out contrasts are
\begin{align}
 \widehat H_{A\mid B\rightarrow B\mid A,u}
 &:=\frac{1}{64}\sum_{i=1}^{64}
 \left[R_{B\mid A,x_i}(q_{x_i}^{(1)})
      -R_{B\mid A,x_i}(q_{x_i}^{\rm ind})\right],\notag\\
 \widehat H_{B\mid A\rightarrow A\mid B,u}
 &:=\frac{1}{64}\sum_{i=1}^{64}
 \left[R_{A\mid B,x_i}(q_{x_i}^{(0)})
      -R_{A\mid B,x_i}(q_{x_i}^{\rm ind})\right].
 \label{eq:tfm-heldout-endpoints}
\end{align}
Here ``direction-held-out'' means that the evaluated opposite-direction
conditional prediction was omitted from the optimization objective for the
same test row.  It does not mean evaluation against new labels, a new model,
or a population of datasets.  Negative $\widehat H$ favors one-direction
reconciliation; positive $\widehat H$ favors the independent joint.

The dataset--split unit, not an individual row or solver call, is the primary
generalization unit.  We report all six unit estimates.  Each interval is a
two-sided percentile interval from 10,000 paired-row bootstrap resamples
within that unit.  It quantifies within-unit row variation and is not a
population interval over datasets or splits.  Three-split dataset means are
descriptive; no cross-dataset population test is performed.

\subsection{Extension results}

All six units passed their numerical status gate.  Every one-direction and
bidirectional method passed on every row: $192/192$ certificates per unit and
$1{,}152/1{,}152$ overall.  Table~\ref{tab:tfm-external-summary} reports both
absolute and relative bidirectional reductions.  The absolute gains are much
larger on Mushroom than on Bank Marketing, so the percentages should not be
read without their risk scales.

\begin{table}[!htbp]
  \centering
  \footnotesize
  \setlength{\tabcolsep}{4pt}
  \caption{Preregistered extension: all $192/192$ instances pass independent
  numerical checks in each unit ($64$ rows, three objectives), for
  $1{,}152/1{,}152$ overall.  Gain is independent-joint minus bidirectional
  risk.  Brackets give within-unit paired-row bootstrap 95\% intervals;
  relative reduction uses the aggregate mean baseline risk.}
  \label{tab:tfm-external-summary}
  \begin{tabular}{@{}lrrrlr@{}}
    \toprule
    Dataset / split & $\overline R_{\rm ind}$ & $\overline R_{\rm bi}$ &
    $\widehat G^{\rm bi}$ & 95\% interval & Reduction \\
    \midrule
    Mushroom / 0 & 0.203929 & 0.004938 & 0.198991 & [0.143780, 0.260780] & 97.58\% \\
    Mushroom / 1 & 0.712894 & 0.009236 & 0.703658 & [0.531375, 0.881553] & 98.70\% \\
    Mushroom / 2 & 0.588489 & 0.014333 & 0.574155 & [0.425745, 0.732887] & 97.56\% \\
    \addlinespace
    Bank Marketing / 0 & 0.013765 & 0.006852 & 0.006913 & [0.002849, 0.012826] & 50.22\% \\
    Bank Marketing / 1 & 0.014419 & 0.006153 & 0.008266 & [0.004651, 0.013017] & 57.33\% \\
    Bank Marketing / 2 & 0.016563 & 0.005444 & 0.011119 & [0.004792, 0.019215] & 67.13\% \\
    \bottomrule
  \end{tabular}
\end{table}

Figure~\ref{fig:tfm-direction-holdout} shows the two direction-held-out
contrasts.  Optimizing $A\mid B$ improves the omitted $B\mid A$ direction on
all three Mushroom splits, with intervals below zero, but the three Bank
Marketing intervals cross zero.  In the reverse analysis, optimizing
$B\mid A$ worsens the omitted $A\mid B$ direction on average in all six
units, with every interval above zero.

\begin{figure}[!htb]
  \centering
  \includegraphics[width=\linewidth]{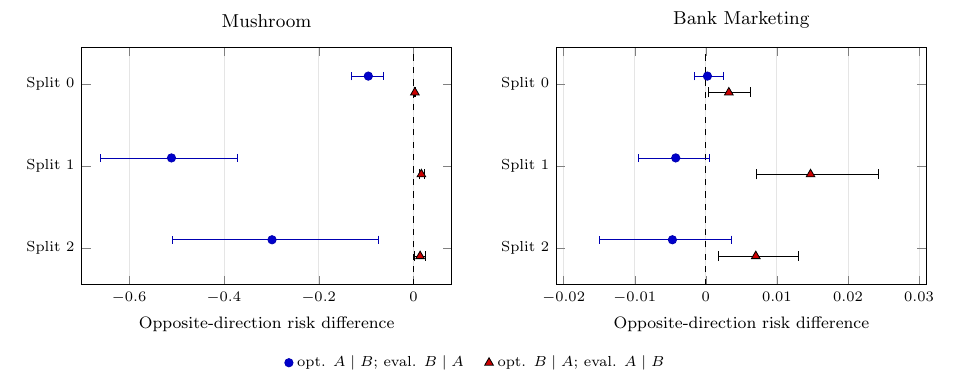}
  \caption{Direction-held-out risk differences relative to the independent
  joint; negative values favor reconciliation on the omitted direction.
  Points average 64 paired rows; bars are 95\% percentile intervals from
  10,000 within-unit paired bootstrap resamples.  Intervals describe row
  variation, not population uncertainty over datasets or splits.
  Panel scales differ.}
  \label{fig:tfm-direction-holdout}
\end{figure}
\FloatBarrier

\subsection{Interpretation}

The direction-held-out analysis distinguishes an accurately solved
projection objective from improved fit in every conditional direction.
On the registered units, the independent checks support the computed
solutions, yet the opposite-direction contrasts depend on both the dataset
and which direction was optimized.  This complements the theory's
finite-access question: approaching a selected projection optimum and
obtaining improvement outside its fitted objective are different goals.
The interpretation remains conditional on the frozen checkpoint, selected
units, and within-unit uncertainty assessment;
Section~\ref{sec:limitations} collects the resulting scope limits.
\section{Discussion and Scope}
\label{sec:limitations}

Finite-access precision and directional fidelity are distinct aspects of
coherent reconciliation.  The theory quantifies the excess risk of
estimating a selected projection from partial access; the case study
examines whether fitting one direction also improves another under fixed
marginals.  Their common conditional-risk framework separates accurate
approximation of an optimum from the consequences of choosing that
optimum.  The output classes and observation protocols differ, so the
case study complements rather than empirically validates the minimax rate.

\paragraph{Statistical and computational scope}
The matching rate allows arbitrary deterministic binary teacher tables,
including probabilities zero and one, under uniform passive-iid exact-table
slots.  Its local bounded proper pairwise student class and budget window
remain substantive restrictions.  The radius
restriction yields the decreasing average signal in
Eq.~\eqref{eq:student-signal-scope}; the result is not a general
fixed-strength high-dimensional reconstruction theorem.  It does not
establish a new rate shape, the cost of teacher incompatibility, an
advantage over directed noncoherent estimation, or a joint-KL/TV guarantee.
Approaching the same-class optimum reduces excess risk, not necessarily
the incompatibility or approximation floors.  The known population design
term in the upper-bound estimator also prevents interpreting query
optimality as a polynomial-time computational guarantee.

\paragraph{Meaning of the empirical findings}
The hard-marginal case study provides complete numerical check coverage
within one checkpoint and the registered units.  Feasibility of the
independent baseline implies exact optimized-objective non-inferiority;
the observed gain magnitude and opposite-direction contrasts are the
empirical questions.  Direction-held-out means omission from a row's
optimization objective, not prediction against new labels.  Intervals
describe paired row variation within a unit, not uncertainty over a
population of datasets or models.  Numerical checks at registered
tolerances are not exact-feasibility proofs for stored floating-point rows.
No predictive-accuracy or general-TFM reliability conclusion follows.

\paragraph{Separate auxiliary evidence and future questions}
The supplementary response-adaptive construction has a non-joint output
and a finite specified comparator class.  Its public-data weighting study
fails the all-dataset gate because Mushroom reverses sign
(Appendix~\ref{app:supplementary-studies}); the failure is retained, not
pooled with the real-TFM case study.  It does not supply an active coherent
query theorem.  Joint-level finite-query precision, or the cost of coherent
rather than directed output under one common observation model, remain
distinct open questions.  Answering either would require additional
theory, not extrapolation of the current rate.
\Needspace{9\baselineskip}
\section*{Conclusion}
\label{sec:conclusion}

We characterize finite-access precision and examine directional fidelity
in coherent projection of black-box conditional predictions.  For arbitrary
deterministic binary teacher tables, matching expected excess-risk bounds
identify the budget, dimension, and radius dependence of the best local
bounded pairwise projection under uniform passive exact-table queries.
The characterization retains its stated radius and budget window.

The fixed-marginal TabPFN-v2 case study examines the complementary question
of what a selected projection fits well.  Independent numerical checks
support the computed solutions, while the asymmetric directional contrasts
show that improvement in a fitted objective need not extend to the omitted
direction.  Together, the results separate estimating a projection
accurately from improving every conditional direction.  Joint-level
precision, coherent active design, and guarantees of predictive-accuracy
improvement remain questions for further study.

\section*{Data availability}
The original datasets used in this study are publicly available through
OpenML~\cite{vanschoren2013openml}.  The case study uses
Car~\cite{dataset_car}, Mushroom~\cite{dataset_mushroom}, and
Bank Marketing~\cite{dataset_bank_marketing}.  The semisynthetic study
also uses Letter Recognition~\cite{dataset_letter}, the morphological
feature subset of Multiple Features~\cite{dataset_mfeat},
Nursery~\cite{dataset_nursery}, and Page Blocks
Classification~\cite{dataset_page_blocks}.  The dataset references provide
original-source DOIs and links to the OpenML records, with the versions
actually used.  The case-study data handling and split design are described
in Section~\ref{sec:experiments}; the semisynthetic use of public data is
described in Appendix~\ref{sec:semisynthetic-experiment}.

\section*{Declaration of generative AI and AI-assisted technologies in the
manuscript preparation process}
During the preparation of this work, the authors used OpenAI Codex (OpenAI),
using GPT-5.6 Sol (model identifier: \texttt{gpt-5.6-sol}), to assist with
manuscript organization, language and clarity revision, LaTeX preparation,
consistency checks, and the preparation and revision of research scripts and
protocol-verification artifacts, including preflight, offline verification,
and tests.  This supported the research workflow, not development of a
standalone software product.  The tool was not a source of experimental
observations; it did not independently select datasets, alter frozen
protocols or checkpoints, authorize executions, or replace stored outputs.
Reported values were derived from frozen execution artifacts and checked
by protocol-defined verifiers.  The authors reviewed and edited the
AI-assisted output as needed, verified the final manuscript and reported
results against the frozen source files and protocol-defined checks, and take
full responsibility for the content of the published article.

\bibliographystyle{elsarticle-num}


\appendix
\renewcommand{\thesection}{\Alph{section}}
\makeatletter
\let\manuscript@seccntformat\@seccntformat
\renewcommand{\@seccntformat}[1]{%
  \ifstrequal{#1}{section}{\appendixname~}{}%
  \manuscript@seccntformat{#1}}
\makeatother
\numberwithin{figure}{section}
\numberwithin{table}{section}
\renewcommand{\thefigure}{\Alph{section}.\arabic{figure}}
\renewcommand{\thetable}{\Alph{section}.\arabic{table}}
\section{A passive exact-table minimax characterization}
\label{sec:minimax-theory}

We isolate a setting in which the statistical cost of projecting an arbitrary
conditional table onto a coherent pairwise model can be characterized up to
constants.  The result is deliberately protocol-specific: it concerns a
uniform passive design with exact deterministic table responses.  It is not an
active-query or computational minimax theorem.

\subsection{Model and risk}

Let $X\in\{-1,+1\}^d$.  A query key is
$z=(j,x_{-j})$, where $j\in[d]$ is the target coordinate and
$x_{-j}\in\{-1,+1\}^{d-1}$ is the context.  The set of all keys is
$\mathcal Z$, with
\[
  N:=|\mathcal Z|=d2^{d-1},
\]
and $\mu$ denotes the uniform distribution on $\mathcal Z$.

The coherent pairwise student has
\[
  m=d+\frac{d(d-1)}2
\]
parameters: one intercept per target and one shared parameter per unordered
pair.  We encode its conditional logit as $a(z)^\top\theta$, where the feature
row $a(z)\in\mathbb R^m$ has the target intercept equal to one, the
$d-1$ edges incident to the target equal to the corresponding context signs,
and every other coordinate equal to zero.  Thus
\[
  q_\theta(X_j=+1\mid X_{-j}=x_{-j})
  =\sigma(a(z)^\top\theta),
  \qquad
  \Theta_R:=\{\theta:\|\theta\|_2\le R\},
\]
where $\sigma(u)=(1+e^{-u})^{-1}$.

The teacher is an arbitrary deterministic table
\[
  p:\mathcal Z\longrightarrow[0,1].
\]
Endpoint values use the KL convention from Section~\ref{sec:problem};
the student probabilities are strictly positive.
For $B$ passive design slots, the learner observes
\[
  Z_1,\ldots,Z_B\overset{\mathrm{i.i.d.}}\sim\mu,
  \qquad p(Z_1),\ldots,p(Z_B).
\]
If the same key appears more than once, its response is the same frozen table
entry.  In particular, a collision is not a fresh noisy label.  The population
risk is
\[
  \mathcal R(\theta;p)
  :=\mathbb E_{Z\sim\mu}
  D_{\mathrm{KL}}
  \left(
    \operatorname{Ber}(p(Z))
    \,\middle\|\,
    \operatorname{Ber}(\sigma(a(Z)^\top\theta))
  \right).
\]
We write
\begin{equation}
  \mathfrak M_{B,\mathrm{iid}}^{\mathrm{agn}}(d,R)
  :=\inf_{\mathcal A}\sup_{p\in\mathcal P_{\rm all}}
  \mathbb E\left[
    \mathcal R(\widehat\theta;p)
    -\inf_{\theta\in\Theta_R}\mathcal R(\theta;p)
  \right],
  \label{eq:minimax-definition}
\end{equation}
where the infimum ranges over randomized learners under the preceding passive
protocol whose output $\widehat\theta$ belongs to $\Theta_R$ almost surely,
the supremum ranges over the deterministic teacher tables specified above,
and the expectation is over the design and learner randomization.

Let $A(u)=\log(1+e^u)$ and define
\[
  G(\theta):=\mathbb E_\mu A(a(Z)^\top\theta),
  \qquad
  b_p:=\mathbb E_\mu[p(Z)a(Z)].
\]
Up to a term $C_p$ independent of $\theta$,
\begin{equation}
  \mathcal R(\theta;p)=G(\theta)-b_p^\top\theta+C_p.
  \label{eq:risk-moment-form}
\end{equation}
Thus the teacher affects its best pairwise projection only through the
$m$-dimensional moment $b_p$.

\begin{lemma}[Uniform-design covariance and curvature]
\label{lem:covariance}
Let $r=d(d-1)/2$ be the number of edge coordinates.  Under the uniform
full-key design,
\begin{equation}
  \Sigma:=\mathbb E_\mu[a(Z)a(Z)^\top]
  =\operatorname{diag}\left(\frac1d I_d,\frac2d I_r\right).
  \label{eq:covariance}
\end{equation}
If $R\sqrt d\le M_0$ and
$\kappa_0:=\sigma(M_0)(1-\sigma(M_0))\in(0,1/4]$, then for every
$\theta\in\Theta_R$,
\begin{equation}
  \frac{\kappa_0}{d}I_m
  \preceq\nabla^2G(\theta)
  \preceq\frac1{2d}I_m.
  \label{eq:hessian-sandwich}
\end{equation}
On the edge subspace, the lower bound improves to
$(2\kappa_0/d)I_r$.
Here $U\preceq V$ is the Loewner order: $V-U$ is positive semidefinite.
\end{lemma}

\begin{proof}
Each row has one nonzero intercept and $d-1$ nonzero incident-edge entries.
The intercept coordinate for target $j$ is active with probability $1/d$,
and an edge is active when either endpoint is queried, with probability
$2/d$.  Every cross moment vanishes by the independent Rademacher context
signs, proving~\eqref{eq:covariance}.  Moreover, abbreviating $a=a(Z)$
inside the expectation,
\[
  \nabla^2G(\theta)
  =\mathbb E_\mu\left[
    \sigma(a^\top\theta)(1-\sigma(a^\top\theta))aa^\top
  \right].
\]
Since $\|a(z)\|_2=\sqrt d$ and $\|\theta\|_2\le R$, the logit has magnitude
at most $M_0$.  Hence its logistic derivative lies in
$[\kappa_0,1/4]$.  Combining this with~\eqref{eq:covariance} gives the
claimed bounds.
\end{proof}

\subsection{Moment estimator: the matching upper bound}

Given the $B$ exact responses, set
\[
  \widehat b_B:=\frac1B\sum_{t=1}^B p(Z_t)a(Z_t)
\]
and output any minimizer
\begin{equation}
  \widehat\theta_B\in\arg\min_{\theta\in\Theta_R}
  \{G(\theta)-\widehat b_B^\top\theta\}.
  \label{eq:moment-estimator}
\end{equation}
The population design term $G$ is treated as known.  This is an information
or query result; it does not assert that the exponential sum defining $G$ can
be evaluated in polynomial time.

\begin{proposition}[Moment-projection upper bound]
\label{prop:moment-upper}
Under the preceding protocol, for every deterministic table
$p:\mathcal Z\to[0,1]$,
\begin{equation}
  \mathbb E\left[
    \mathcal R(\widehat\theta_B;p)
    -\inf_{\theta\in\Theta_R}\mathcal R(\theta;p)
  \right]
  \le
  \min\left\{
    2R\sqrt{\frac dB},
    \frac{d^2}{2\kappa_0B}
  \right\}.
  \label{eq:upper-bound}
\end{equation}
\end{proposition}

\begin{proof}
Let $b=b_p$, $e=\widehat b_B-b$, and let $\theta^*$ minimize
$G(\theta)-b^\top\theta$ on $\Theta_R$.  Optimality in
\eqref{eq:moment-estimator} gives
\[
  \mathcal R(\widehat\theta_B;p)-\mathcal R(\theta^*;p)
  \le e^\top(\widehat\theta_B-\theta^*)
  \le2R\|e\|_2.
\]
Because $\|p(Z)a(Z)\|_2^2\le d$,
\[
  \mathbb E\|e\|_2^2
  =\frac1B\mathbb E\|p(Z)a(Z)-b\|_2^2
  \le\frac dB.
\]
Jensen's inequality gives the first term in~\eqref{eq:upper-bound}.

By Lemma~\ref{lem:covariance}, the perturbed objective in
\eqref{eq:moment-estimator} is $\lambda$-strongly convex on $\Theta_R$ with
$\lambda=\kappa_0/d$.  Constrained optimality and strong convexity give
\[
  \mathcal R(\widehat\theta_B;p)-\mathcal R(\theta^*;p)
  \le
  e^\top(\widehat\theta_B-\theta^*)
  -\frac\lambda2\|\widehat\theta_B-\theta^*\|_2^2
  \le\frac{\|e\|_2^2}{2\lambda}.
\]
Taking expectations gives the second term.
\end{proof}

\subsection{Matching lower bound}

The full statement of Theorem~\ref{thm:passive-minimax} appears in
Section~\ref{sec:passive-protocol}.  Proposition~\ref{prop:moment-upper}
already bounds the risk uniformly over $\mathcal P_{\rm all}$.  To prove
the lower bound, fix $\gamma:=1/4$ throughout the construction below and
restrict to
\[
 \mathcal P_{1/4}
 :=\{p:\mathcal Z\to[1/4,3/4]\}
 \subset\mathcal P_{\rm all}.
\]
With the same learner class and observation protocol, enlarging the teacher
class cannot decrease the minimax risk: for every learner its worst-case
excess risk over $\mathcal P_{\rm all}$ is at least that over
$\mathcal P_{1/4}$, and taking the infimum preserves the inequality.
The two-valued tables below lie in this subclass.  Hence their lower bound,
together with the uniform upper bound, proves the full-class statement
without changing the student radius or budget window.  Since $\gamma$
is fixed, the theorem constants depend only on $M_0$.

The same construction also retains the restricted-class result: for any
fixed $\gamma\in(0,1/4]$, replacing the fixed amplitude by $\gamma$ gives
matching bounds over
$\mathcal P_\gamma:=\{p:\mathcal Z\to[1/2-\gamma,1/2+\gamma]\}$,
with constants depending on $M_0$ and $\gamma$.  The full-class theorem
needs only the fixed choice $\gamma=1/4$.

\begin{proof}
The upper bound is Proposition~\ref{prop:moment-upper}.  We prove the
full-class lower bound using the preceding fixed-amplitude subclass in
seven steps.

\paragraph{Step 1: Edge code and frozen deterministic tables}
Restrict the hard family to the $r=d(d-1)/2$ edge coordinates and set every
intercept perturbation to zero.  A Varshamov--Gilbert
code~\cite{tsybakov2009introduction} gives
$\Omega\subset\{-1,+1\}^r$ such that
\[
  \log|\Omega|\ge c_{\mathrm{VG}}r,
  \qquad
  d_{\mathrm H}(\omega,\omega')\ge r/8
  \quad(\omega\ne\omega').
\]
Draw $\omega$ uniformly from $\Omega$.  Unless stated otherwise, probabilities
in the lower-bound argument are over this codeword prior, the conditionally
drawn frozen table, the i.i.d. design slots, and learner randomization.
Let
\[
  \delta=c_1\sqrt{d/B},
  \qquad t_\omega=\delta\omega,
\]
where $c_1>0$ will be chosen sufficiently small.
Whenever $t_\omega$ is combined with a full $m$-dimensional parameter or
moment vector, it denotes its zero-intercept embedding in $\mathbb R^m$.
Let $a_E(z)$ denote the edge subvector of $a(z)$.  Conditional on $\omega$,
independently generate
one sign for every key,
\begin{equation}
  S_z\in\{-1,+1\},
  \qquad
  \mathbb E[S_z\mid\omega]
  =c_2a_E(z)^\top t_\omega,
  \label{eq:hard-sign-prior}
\end{equation}
for a fixed absolute $c_2>0$, and then freeze the exact table
\begin{equation}
  p_{\omega,S}(z)=\frac12+\gamma S_z.
  \label{eq:hard-table}
\end{equation}
Since a key activates at most $d-1$ edge coordinates and
$B\ge Cd^3$, choosing $c_1/\sqrt C$ sufficiently small ensures
$|c_2a_E(z)^\top t_\omega|\le1/2$.  Thus~\eqref{eq:hard-sign-prior} is valid,
and every realization of~\eqref{eq:hard-table} is a deterministic teacher in
$\mathcal P_{1/4}\subset\mathcal P_{\rm all}$.  The randomness is a prior
over frozen tables, not response
noise at query time.

\paragraph{Step 2: Collision-aware transcript information}
Let $P_\omega^{(B)}$ be the transcript law after mixing over $S$.  Given the
key sequence, observing the exact responses is equivalent to observing
$S_z$ once for each unique queried key.  Hence
\[
  D_{\mathrm{KL}}(P_\omega^{(B)}\|P_{\omega'}^{(B)})
  \le
  \frac{Bc_2^2}{1-\rho^2}
  \mathbb E_\mu
  [a_E(Z)^\top(t_\omega-t_{\omega'})]^2,
\]
where $\rho=1/2$ is the preceding bias envelope.  Repeated keys can only
remove terms from this upper bound.  By~\eqref{eq:covariance},
\[
  D_{\mathrm{KL}}(P_\omega^{(B)}\|P_{\omega'}^{(B)})
  \le
  \frac{2Bc_2^2}{d(1-\rho^2)}
  \|t_\omega-t_{\omega'}\|_2^2
  \le C_1c_1^2r.
\]
Choose $c_1$ so that this is at most
$\frac18\log|\Omega|$.  Fano's inequality~\cite{tsybakov2009introduction}
then implies that every transcript
decoder satisfies, for all sufficiently large $r$,
\begin{equation}
  \Pr(\widehat\omega\ne\omega)\ge\frac34.
  \label{eq:fano-error}
\end{equation}

\paragraph{Step 3: The realized table stays near its code moment}
Set
\[
  b_0:=\mathbb E_\mu\left[\frac12a(Z)\right],
  \qquad K:=\gamma c_2.
\]
For the realized table, define
$b_{\omega,S}:=b_{p_{\omega,S}}
=\mathbb E_\mu[p_{\omega,S}(Z)a(Z)]$.
The prior-mean moment shift is edge-only and equals
\begin{equation}
  h_\omega
  :=\mathbb E[b_{\omega,S}\mid\omega]-b_0
  =\frac{2K}{d}t_\omega.
  \label{eq:mean-moment-shift}
\end{equation}
Write
$e_{\omega,S}:=b_{\omega,S}-b_0-h_\omega$.  Independence across the frozen
table entries and $\|a(z)\|_2^2=d$ give
\[
  \mathbb E\|e_{\omega,S}\|_2^2\le\frac{\gamma^2d}{N}.
\]
Fix a sufficiently small $\varepsilon_0>0$.  Since
$\delta=c_1\sqrt{d/B}$ and $B\le N^{1/3}$, Markov's inequality implies,
after increasing $d_0$ if needed,
\begin{equation}
  \Pr\left(
    \|e_{\omega,S}\|_2
    \le\varepsilon_0K\kappa_0\delta
  \right)\ge\frac{63}{64}.
  \label{eq:good-table}
\end{equation}
This is the good-table event.

\paragraph{Step 4: Common ideal geometry}
Define
\[
  \Phi(\theta):=G(\theta)-G(0)-b_0^\top\theta.
\]
Then $\Phi(0)=0$, $\nabla\Phi(0)=0$, and its Hessian obeys
Lemma~\ref{lem:covariance}.  Up to a table-dependent additive constant, the
realized objective is
\begin{equation}
  F_{\omega,S}(\theta)
  =\Phi(\theta)-(h_\omega+e_{\omega,S})^\top\theta.
  \label{eq:realized-objective}
\end{equation}
The ideal objective is $F_\omega^0(\theta)=\Phi(\theta)-h_\omega^\top\theta$.

We first treat the interior regime.  Let $\theta_\omega^0$ be the unique
minimizer of $F_\omega^0$ over $\Theta_R$.  For any boundary point
$\|\theta\|_2=R$, integrating the lower Hessian bound along the segment
$s\theta$ gives
\[
  \theta^\top\nabla F_\omega^0(\theta)
  \ge\frac{\kappa_0}{d}R^2-R\|h_\omega\|_2.
\]
A constrained boundary minimizer would instead require this radial derivative
to be nonpositive.  Therefore, if $\|h_\omega\|_2<\kappa_0R/d$, the ideal
minimizer is interior.  Symmetry of the Rademacher contexts makes every
intercept objective strictly convex and stationary at zero, so
$\theta_\omega^0$ has zero intercept coordinates.  The edge-subspace lower
curvature and its first-order condition then yield
\[
  \|\theta_\omega^0\|_2
  \le\frac{d}{2\kappa_0}\|h_\omega\|_2
  =\frac K{\kappa_0}\delta\sqrt r.
\]
Consequently there is a constant $c_{\mathrm{in}}>0$ such that
\begin{equation}
  \delta d\le c_{\mathrm{in}}R
  \label{eq:interior-condition}
\end{equation}
places every ideal center inside the ball of radius $R/2$.

The upper Hessian bound gives, for distinct codewords,
\begin{align}
  \|\theta_\omega^0-\theta_{\omega'}^0\|_2
  &\ge2d\|h_\omega-h_{\omega'}\|_2 \\
  &=4K\|t_\omega-t_{\omega'}\|_2 \\
  &\ge4K\delta\sqrt{r/2}
  =:\Delta_{\mathrm{in}}.
  \label{eq:center-separation}
\end{align}
Let $\theta_{\omega,S}^*$ minimize the realized objective.  Strong convex
variational inequalities for two objectives differing by a linear term imply
\begin{equation}
  \|\theta_{\omega,S}^*-\theta_\omega^0\|_2
  \le\frac d{\kappa_0}\|e_{\omega,S}\|_2.
  \label{eq:optimizer-perturbation}
\end{equation}
On the good-table event, choosing $\varepsilon_0$ sufficiently small makes
this at most $\Delta_{\mathrm{in}}/8$.

\paragraph{Step 5: Interior decoding implies $d^2/B$ excess risk}
Decode a learner output by its nearest ideal center.  A decoding error implies
$\|\widehat\theta-\theta_\omega^0\|_2\ge\Delta_{\mathrm{in}}/2$ and, by
\eqref{eq:optimizer-perturbation},
\[
  \|\widehat\theta-\theta_{\omega,S}^*\|_2
  \ge3\Delta_{\mathrm{in}}/8
\]
on the good-table event.  Since the realized objective is
$\kappa_0/d$-strongly convex on the ball,
\begin{equation}
  F_{\omega,S}(\widehat\theta)
  -F_{\omega,S}(\theta_{\omega,S}^*)
  \ge c_2'\frac{\delta^2r}{d}
  \ge c_3'\frac{d^2}{B}.
  \label{eq:interior-risk}
\end{equation}

\paragraph{Step 6: Boundary decoding implies $R\sqrt{d/B}$ excess risk}
Let
\[
  u_\omega:=\frac{(0_d,t_\omega)}{\|t_\omega\|_2}\in\mathbb R^m,
  \qquad
  H:=\|h_\omega\|_2=\frac{2K}{d}\delta\sqrt r.
\]
Thus $H\asymp K\delta$, where the comparison constants are absolute (and in
particular do not depend on $d$, $B$, or $R$).  Code separation gives
$u_\omega^\top u_{\omega'}\le3/4$ for $\omega\ne\omega'$.  Define the decoder by
$\widehat\omega\in\operatorname*{arg\,max}_{\nu\in\Omega}
u_\nu^\top\widehat\theta$, with a fixed lexical tie rule.  If this decoder
is wrong, elementary
halfspace geometry gives
\begin{equation}
  u_\omega^\top\widehat\theta
  \le R\sqrt{\frac{1+3/4}{2}}
  =R\sqrt{7/8}.
  \label{eq:angular-loss}
\end{equation}

Convexity and $\nabla\Phi(0)=0$ imply $\Phi\ge0$, while smoothness gives
$\Phi(Ru_\omega)\le R^2/(4d)$.  Using $Ru_\omega$ as a feasible comparator
in~\eqref{eq:realized-objective}, every decoder-error output satisfies
\begin{align}
  F_{\omega,S}(\widehat\theta)
  -\min_{\theta\in\Theta_R}F_{\omega,S}(\theta)
  &\ge
  \left(1-\sqrt{7/8}\right)HR
  -\frac{R^2}{4d}
  -2R\|e_{\omega,S}\|_2.
  \label{eq:boundary-comparator}
\end{align}
Hence there is a constant $C_{\mathrm{bd}}>0$ such that, on the good-table
event,
\begin{equation}
  \delta d\ge C_{\mathrm{bd}}R
  \quad\Longrightarrow\quad
  F_{\omega,S}(\widehat\theta)
  -F_{\omega,S}(\theta_{\omega,S}^*)
  \ge c_4'R\delta
  =c_5'R\sqrt{d/B}.
  \label{eq:boundary-risk}
\end{equation}

\paragraph{Step 7: Fano, Yao, and the transition band}
For either decoder,~\eqref{eq:fano-error} and~\eqref{eq:good-table} imply
\[
  \Pr(\widehat\omega\ne\omega
  \text{ and the table is good})
  \ge\frac34-\frac1{64}>\frac12.
\]
Combining this event with~\eqref{eq:interior-risk} or
\eqref{eq:boundary-risk} yields the corresponding Bayes expected excess-risk
lower bound.  The prior is supported entirely on deterministic exact tables
in $\mathcal P_{1/4}\subset\mathcal P_{\rm all}$,
so Yao's principle~\cite{yao1977probabilistic} transfers the Bayes lower bound
to randomized learners in
the minimax problem~\eqref{eq:minimax-definition}.

It remains to connect the two sufficient regimes.  Write
$\delta_B=c_1\sqrt{d/B}$, let $U=N^{1/3}$, and enlarge
$C_{\mathrm{bd}}$ if necessary so that $C_{\mathrm{bd}}>c_{\mathrm{in}}$.
The boundary and interior conditions respectively hold for
\[
  B\le B_{\mathrm{bd}}
  :=\frac{c_1^2d^3}{C_{\mathrm{bd}}^2R^2},
  \qquad
  B\ge B_{\mathrm{in}}
  :=\frac{c_1^2d^3}{c_{\mathrm{in}}^2R^2}.
\]
These thresholds differ by the fixed factor
$(C_{\mathrm{bd}}/c_{\mathrm{in}})^2$.  If
$B_{\mathrm{in}}\le Cd^3$, the interior argument already covers the
declared window.  Otherwise set $B_*:=\lceil B_{\mathrm{in}}\rceil$.
For integer slot budgets $B\in[B_{\mathrm{bd}},B_*]$, the target
minimum is at most $(C_{\mathrm{bd}}/c_1)R^2/d$, while the interior
lower bound at $B_*$ is a constant multiple of $R^2/d$; rounding changes
only that constant for sufficiently large $d$.  Since the minimax risk is
nonincreasing in the number of passive slots, the bound at $B_*$ transfers
to every smaller budget in this transition band.

For completeness, fix any constant
$L>4(C_{\mathrm{bd}}/c_{\mathrm{in}})^2$.  If $B_*\le LU$, the
interior construction is valid at $B_*$ even if $B_*>U$: its Fano and bias
bounds require only $B_*\ge Cd^3$, and its good-table probability requires
$\sqrt{B_*/N}$ to be sufficiently small, which follows from
$B_*\le L N^{1/3}$ by increasing $d_0$.  If $B_*>LU$, then for all
sufficiently large $d$ one has $B_{\mathrm{in}}>LU-1$, hence
$\delta_Ud> C_{\mathrm{bd}}R$; the boundary argument therefore covers
every $B\le U$.  This exhausts the transition and window-endpoint cases
without changing the stated theorem window or rate, and proves the lower
bound in~\eqref{eq:matching-rate}.
\end{proof}

\begin{remark}[Scope]
Theorem~\ref{thm:passive-minimax} counts i.i.d. design slots, not guaranteed
distinct keys.  Caching a repeated exact response does not authorize drawing a
replacement key.  The theorem also does not cover arbitrary design
distributions, active querying, high-probability matching rates, or
polynomial-time evaluation of $G$.
\end{remark}
\section{Proof of the unstructured query barrier}
\label{app:unstructured-proof}

\begin{proof}[Proof of Proposition~\ref{thm:unstructured-barrier}]
Put the uniform prior on the $N_{\rm ctx}=2^{d-1}$ independent signs in
Eq.~\eqref{eq:unstructured-hard-family}.  Give the context vector $Z$ a
uniform marginal.  For each sign assignment, combining this marginal with
the specified conditional of $Y$ constructs a compatible joint law.

Fix any adaptive querying algorithm and condition on its transcript,
including its internal randomization.  An exact response determines the sign
at the queried context.  The signs at all unqueried contexts remain
independent and uniform under the prior.  If the algorithm predicts
$u\in(0,1)$ at one such context, its conditional Bayes loss is
\begin{align*}
  \frac12 D_{\mathrm{KL}}\!\left(
    \operatorname{Ber}(1/2+\gamma)\,\middle\|\,\operatorname{Ber}(u)
  \right)
  +\frac12 D_{\mathrm{KL}}\!\left(
    \operatorname{Ber}(1/2-\gamma)\,\middle\|\,\operatorname{Ber}(u)
  \right).
\end{align*}
The average log loss is minimized at $u=1/2$, where it equals
$\kappa(\gamma)$.  This lower bound applies to the conditionals induced by
any coherent joint output; requiring coherence cannot reduce it.

At most $B$ distinct contexts can have been queried, so at least
$(N_{\rm ctx}-B)_+$ remain unqueried.  Their contribution to the uniform-context
risk is at least $(1-B/N_{\rm ctx})_+\kappa(\gamma)$ on average under the prior.
Every other contribution is nonnegative.  Averaging over transcripts and
then taking the worst sign assignment proves the first inequality for every
randomized algorithm.  Finally, Pinsker's inequality for Bernoulli
distributions~\cite{tsybakov2009introduction}
gives $\kappa(\gamma)\ge2\gamma^2$; multiplication by the
nonnegative unqueried fraction proves the second inequality and the stated
query threshold.
\end{proof}
\section{Fixed-design structured bound}
\label{app:structured-upper-proof}

Equation~\eqref{eq:structured-upper} is a standard bounded-loss uniform
convergence bound.  We give the specialization to fix its constants, using
symmetrization, contraction, and bounded differences
\cite[Sections~2--3]{bartlett2002rademacher}.

For $L_\theta(z):=D_{\rm KL}(\operatorname{Ber}(p(z))\|
\operatorname{Ber}(\sigma(a(z)^\top\theta)))$, write
$v=a(z)^\top\theta$ and $A(v)=\log(1+e^v)$.  Then
$L_\theta=A(v)-p(z)v-H(p(z))$, $|v|\le M$, $0\le L_\theta\le C_M$,
and $|\partial L_\theta/\partial v|\le1$.
In the one-sided Rademacher convention (no absolute value inside the
supremum), centering and scalar contraction give
\[
 \mathbb E\sup_{\theta\in\Theta_R}\frac1B\sum_i\epsilon_i
       \{L_\theta(Z_i)-L_0(Z_i)\}
 \le\frac RB\mathbb E\left\|\sum_i\epsilon_i a(Z_i)\right\|_2
 \le R\sqrt{d/B}.
\]
Here $\epsilon_i$ are independent Rademacher signs and
$\|a(Z_i)\|_2^2=d$.
Each one-sided uniform deviation between $R_\mu$ and empirical risk
$\widehat R_B$ therefore has mean at most $2R\sqrt{d/B}$ and bounded
differences $C_M/B$.  McDiarmid's inequality and a union bound make both
deviations at most
$2R\sqrt{d/B}+C_M\sqrt{\log(2/\delta)/(2B)}$ with probability $1-\delta$.
Empirical optimality adds the two deviations, proving
Eq.~\eqref{eq:structured-upper}.

For Eq.~\eqref{eq:bounded-residual}, put $v_0=a(z)^\top\theta_0$.
The Bernoulli KL is the Bregman remainder
$A(v_0)-A(v_0+r)+A'(v_0+r)r\le r^2/8$, since $A''\le1/4$.
Average and compare with $\theta_0$.  This is a smoothness calculation, not
an additional reconstruction theorem.
\section{Compatible rank boundary}
\label{app:rank-proof}

To verify Eq.~\eqref{eq:compatible-rank}, put the uniform prior on the
Euclidean ball $\Theta_R\subset\mathbb R^m$.  An exact compatible response
reveals $a(z)^\top\theta$ after taking its logit.  Conditional on any
transcript whose measured rows have rank $k$, rotate coordinates so that the
first $k$ directions span those rows.  The unmeasured component remains
uniform on an $(m-k)$-dimensional ball whose squared radius is $R^2$ minus
the squared norm of the observed projection.  The same statement remains
true when each new row is chosen from previous observations: a new
independent row can reveal only one direction of this rotationally symmetric
residual ball.

For a uniform ball in $n$ dimensions with radius $s$, each coordinate has
second moment $s^2/(n+2)$ and the total second moment is
$ns^2/(n+2)$.  After $k$ independent measurements, iterating this identity
shows that the prior-averaged posterior squared error is
$R^2(m-k)/(m+2)$.  Repeated or dependent measurements can only leave more
unmeasured directions.  Equivalently, pad a shorter transcript with
independent measurement directions before taking the Bayes average;
additional information can only reduce posterior error.  Thus $k\le B$
yields the lower bound
$R^2(m-B)_+/(m+2)$ for any estimator, including randomized estimators.
The uniform full-key covariance obeys $\Sigma\succeq I_m/d$ and the
Bernoulli log-partition has curvature at least $\kappa_0$ on the bounded
student ball.  Therefore conditional KL from a compatible teacher
$\theta$ to an output $\widehat\theta$ is at least
$\kappa_0\|\widehat\theta-\theta\|_2^2/(2d)$.
The Bayes bound implies the stated worst-case bound.  Finally, the legal
rows span $\mathbb R^m$: a vector orthogonal to every row must have zero
intercept and edge coefficients by averaging each queried target over all
Rademacher contexts.  Hence $m$ linearly independent rows exist and their
exact logits determine $\theta$ uniquely.
\section{Hard-marginal optimality and certificates}
\label{app:transport-proof}

\begin{proposition}[Direction-weighted hard-marginal transport]
\label{prop:tfm-transport}
Suppose the direct marginals are strictly positive and normalized,
$c_{ab}^{(\omega)}>0$, and
$\underline q_{ab}=\alpha p_A(a\mid x)p_B(b\mid x)$ for $0<\alpha<1$.
For every fixed $\omega\in[0,1]$, Eq.~\eqref{eq:tfm-transport} has a unique
minimizer.  For any dual vector $\nu$ with
$\mathbf y:=\mathsf M^\top\nu>0$ componentwise, its cell response is
\[
 q_{ab}^{(\omega)}(\nu)
 =\max\{\underline q_{ab},c_{ab}^{(\omega)}/y_{ab}\}.
\]
For any exactly feasible $\mathbf q$,
\begin{equation}
 \begin{aligned}
 0
 &\le \mathcal J_x^{(\omega)}(\mathbf q)
       -\mathcal J_x^{(\omega)}(\mathbf q_x^{\star,(\omega)})
 \le G_{\mathrm{PD}}^{(\omega)}(\mathbf q,\nu),\\
 0
 &\le \mathcal J_x^{(\omega)}(\mathbf q)
       -\mathcal J_x^{(\omega)}(\mathbf q_x^{\star,(\omega)})
 \le G_{\mathrm{FW}}^{(\omega)}(\mathbf q).
 \end{aligned}
 \label{eq:tfm-certificate-bounds}
\end{equation}
where $G_{\mathrm{PD}}^{(\omega)}$ is the cellwise Fenchel primal--dual gap
and $G_{\mathrm{FW}}^{(\omega)}$ is the feasible Frank--Wolfe gap.  Projecting
$\nabla\mathcal J_x^{(\omega)}(\mathbf q)$ orthogonally off
$\operatorname{range}(\mathsf M^\top)$ does not change the latter gap.
\end{proposition}

\begin{proof}
Use standard convex duality~\cite[Chapter~5]{boyd2004convex} and the
Frank--Wolfe gap~\cite[Section~2]{jaggi2013revisiting}, with $x,\omega$
suppressed and cells vectorized.  Equation~\eqref{eq:tfm-coefficients}
already gives the objective reduction; the probability floor keeps $c_i>0$
also at $\omega=0,1$.  The independent joint is strictly above
$\underline{\mathbf q}$ and has the required marginals.  Thus $\mathcal P_x$
is nonempty and compact, and
$\nabla^2\mathcal J=\operatorname{diag}(c_i/q_i^2)\succ0$ gives uniqueness.

For $y=\mathsf M^\top\nu>0$, scalar minimization gives the stated clipped
response.  With $f_i(u):=-c_i\log u$, define
\[
 d_i(y_i):=\min_{u\ge\underline q_i}\{f_i(u)+y_i u\},
 \qquad D(\nu):=\sum_i d_i(y_i)-\nu^\top\mathbf r.
\]
Weak duality gives $D(\nu)\le\mathcal J(\mathbf q^\star)$.
The strictly feasible independent joint supplies Slater's condition;
KKT stationarity is
$y_i=c_i/q_i^\star+\lambda_i>0$, with $\lambda_i\ge0$ for the lower bounds.
The dual gradient is $\mathsf M\mathbf q(\nu)-\mathbf r$.  Away from cell
switches its Hessian is
\[
 -\mathsf M\operatorname{diag}\!\left\{
 c_i y_i^{-2}\mathbf 1_{\{q_i(\nu)>\underline q_i\}}
 \right\}\mathsf M^\top;
\]
the active-set limits give generalized Hessian elements at switches.

For exactly feasible $\mathbf q$, define
\begin{align*}
 G_{\rm PD}(\mathbf q,\nu)
 &:=\sum_i\{f_i(q_i)+y_iq_i-d_i(y_i)\}
   =\mathcal J(\mathbf q)-D(\nu),\\
 G_{\rm FW}(\mathbf q)
 &:=\max_{\mathbf s\in\mathcal P_x}
 \langle\nabla\mathcal J(\mathbf q),\mathbf q-\mathbf s\rangle.
\end{align*}
The first equality uses
$y^\top\mathbf q=\nu^\top\mathbf r$.  Weak duality and the standard
Frank--Wolfe bound give Eq.~\eqref{eq:tfm-certificate-bounds}.
Finally, if $\mathbf g-\mathbf g_\perp=\mathsf M^\top\mathbf v$, then every
feasible pair satisfies
\[
 \langle\mathbf g-\mathbf g_\perp,\mathbf q-\mathbf s\rangle
 =\mathbf v^\top(\mathsf M\mathbf q-\mathsf M\mathbf s)=0.
\]
Thus removing the equality-gradient component leaves the gap unchanged.
\end{proof}
\section{TFM development sequence and Car row-wise distribution}
\label{app:tfm-development-sequence}

Table~\ref{tab:tfm-sequence} retains the registered development failures that
preceded the confirmatory Car anchor.  T2 changed the numerical method rather
than retrospectively changing either earlier decision.
Figure~\ref{fig:tfm-paired-risks} shows the individual conditional-risk
reductions for the 64 confirmatory test rows.

\begin{table}[H]
  \centering
  \caption{Registered TFM development sequence.}
  \label{tab:tfm-sequence}
  \small
  \begin{tabular}{@{}l>{\raggedright\arraybackslash}p{0.34\linewidth}>{\raggedright\arraybackslash}p{0.48\linewidth}@{}}
    \toprule
    Stage & Intervention & Diagnostic and registered outcome \\
    \midrule
    T0 & Conditional-only projection & Mean direct-marginal TV drift $0.207$; failed marginal-fidelity gate \\
    T1 & Hard direct marginals with SLSQP & Only $8/64$ rows passed the main numerical gates; failed numerical certificate \\
    T2 & Hard direct marginals with transport solver & $64/64$ rows passed independent numerical checks on a new split; passed confirmatory gates \\
    \bottomrule
  \end{tabular}
\end{table}

\begin{figure}[H]
  \centering
  \includegraphics[width=0.82\linewidth]{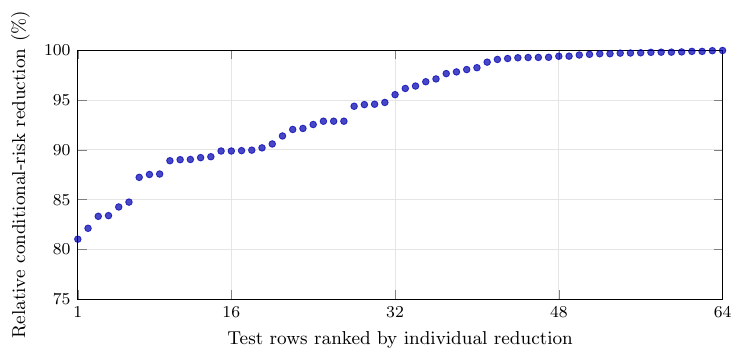}
  \caption{Row-wise relative reduction in paired direction-balanced
  conditional risk for all 64 test rows on one preregistered TabPFN-v2/Car
  split. Rows are sorted by their own reduction; the vertical axis is limited
  to 75--100\% to show within-split spread. Both risks use the same four
  frozen oracle predictions used in reconstruction, not held-out conditional
  predictions. Direct marginals are fixed by construction.}
  \label{fig:tfm-paired-risks}
\end{figure}
\section{Protocol and estimand distinctions}
\label{app:protocol-ledger}

Table~\ref{tab:supplementary-claim-evidence-bridge} records the observation
laws, charged budgets, and evaluation targets.  It expands the main-task
summary in Table~\ref{tab:claim-evidence-bridge}.  Exact frozen responses,
fresh noisy responses, and the TFM's fixed predictions define different
information structures; results are interpreted within their own rows.

\begingroup
\footnotesize
\setlength{\tabcolsep}{4pt}
\setlength{\LTcapwidth}{\linewidth}
\begin{longtable}{@{}>{\raggedright\arraybackslash}p{0.17\linewidth}
                   >{\raggedright\arraybackslash}p{0.37\linewidth}
                   >{\raggedright\arraybackslash}p{0.40\linewidth}@{}}
\caption{Observation protocols and evaluation targets. Rows define distinct
tasks, not interchangeable estimates of the passive minimax rate.}
\label{tab:supplementary-claim-evidence-bridge}\\
\toprule
Study & Observation and charged budget & Output, target, and scope \\
\midrule
\endfirsthead
\multicolumn{3}{@{}l}{\tablename~\thetable\ (continued)}\\
\toprule
Study & Observation and charged budget & Output, target, and scope \\
\midrule
\endhead
\midrule
\multicolumn{3}{r@{}}{Continued on next page}\\
\endfoot
\bottomrule
\endlastfoot
Unstructured coverage &
Exact probabilities; at most $B$ adaptive or nonadaptive context queries. &
Coherent joint; uniform-context conditional KL for a compatible teacher
(Proposition~\ref{thm:unstructured-barrier}). \\
\addlinespace
Passive pairwise projection &
Exact frozen probabilities; $B$ uniform iid slots, with repeats charged. &
Bounded pairwise joint; expected same-table best-projection excess KL
within Theorem~\ref{thm:passive-minimax}'s radius and budget window. \\
\addlinespace
Fixed-pool recovery &
Quantized logits; prefixes $128,256,512$; $4{,}096$ reference calls per
target charged separately. &
Least-squares parameters; distance to the frozen full-pool target in
Eq.~\eqref{eq:p2-fixed-pool-target}. \\
\addlinespace
Restricted auxiliary theorem &
Fresh Rademacher logit noise; 18 diagnostic and two local calls. &
Non-joint table; expected full-key loss and E-AUC; 16 static allocations
and response-independent mixtures (Theorem~\ref{thm:restricted-adaptivity}). \\
\addlinespace
Gaussian comparison &
Fresh Gaussian logit noise, standard deviation $0.05$; 18+2 calls. &
Non-joint table E-AUC; two frozen families and two endpoints, with a
different noise law from Theorem~\ref{thm:restricted-adaptivity}. \\
\addlinespace
Semisynthetic bridge &
Public-data context structure and weights; generated oracle with fresh
Gaussian query noise. &
Weighted table E-AUC and minimum static-minus-adaptive contrasts;
the all-dataset gate fails. \\
\addlinespace
Hard-marginal TFM study &
Frozen direct marginals and both conditional families; four prediction
tasks per unit; reconciliation adds no model calls. &
Hard-marginal joint; balanced and opposite-direction risks against the
same rows' frozen predictions, not new labels. \\
\end{longtable}
\endgroup
\section{Auxiliary conditional-table task and construction}
\label{app:auxiliary-query-design}
\label{app:adaptive-families}

This supplementary task compares response-adaptive and response-independent
allocation under a shared learner with a conditional-table output.
Classical active data selection provides background
\cite{mackay1992information,cohn1996active}.
The finite construction below fixes its families, noisy observations, and
16-policy static comparator; Theorem~\ref{thm:restricted-adaptivity}
states the separation and Appendix~\ref{app:adaptive-proof} proves it.

\subsection{Task and loss}

For each full-conditional key $k$ of five binary variables, let $\ell(k)$ denote the teacher logit and
$\widehat\ell_B(k)$ the learner's reported logit at budget $B$.  Define the
auxiliary full-key table loss
\begin{equation}
  \widetilde R_{\rm full}(\widehat\ell_B;\ell)
  :=\frac1{80}\sum_{k=1}^{80}
  D_{\mathrm{KL}}\!\left(
    \operatorname{Ber}(\sigma(\ell(k)))\,\middle\|\,
    \operatorname{Ber}(\sigma(\widehat\ell_B(k)))
  \right),
  \label{eq:auxiliary-table-risk}
\end{equation}
where the 80 keys comprise five targets and 16 contexts per target.  For a
fixed family $F$ and scalar amplitude $\tau$, set
\begin{equation}
  \mathrm{E\text{-}AUC}_{F,\tau}(\pi)
  :=\frac12\sum_{B\in\{18,20\}}
  \mathbb E_{Z,\,\xi}^{\pi}
  \widetilde R_{\rm full}(\widehat\ell_B^{\pi};\ell^{F,Z,\tau}),
  \label{eq:auxiliary-eauc}
\end{equation}
where the superscript indicates that the expectation includes any policy
randomization, in addition to the uniform hidden regime $Z$ and fresh query
noise $\xi$.  E-AUC is therefore the registered two-budget, equal-weight average of
expected full-key table losses at 18 and 20 charged calls; smaller values
indicate better conditional-table prediction.  This table loss does not
require the reported conditionals to be compatible and is \emph{not} the coherent-joint risk
$R_\mu(q;\mathcal O)$ from Eq.~\eqref{eq:conditional-kl-risk}.

\subsection{Restricted construction and shared learner}

Let $X_0,\ldots,X_4\in\{-1,+1\}$ and evaluate uniformly over all five
single-target full-conditional keys and all 16 contexts per target.
Fix $L=1/2$, $a=\delta=\nu=1/20$, and
$\tau\in[19/100,1/5]$.  The known family $F$ is either $A$ or $B$;
within a family the hidden regime $Z$ is uniform on $\{-1,+1\}$.
Put $\psi(x)=x_1x_2x_3x_4$ and define the nonzero teacher logits by
\begin{align}
  \ell_0^{A,Z,\tau}(x_{1:4})
    &=a x_1+\delta Z\psi(x)+\tau\phi_Z^A(x),
  &\ell_1^{A,Z,\tau}(x_{-1})&=a x_0,
  \label{eq:a1-family-a}\\
  \ell_0^{B,Z,\tau}(x_{1:4})
    &=a x_2+\delta Z\psi(x)+\tau\phi_Z^B(x),
  &\ell_2^{B,Z,\tau}(x_{-2})&=a x_0,
  \label{eq:a1-family-b}
\end{align}
with every other target logit zero and
\begin{equation}
  \phi_Z^A(x)=\mathbf 1_{\{x_1=Z\}}\frac{(x_2-x_3)x_4}{2},
  \qquad
  \phi_Z^B(x)=\mathbf 1_{\{x_2=Z\}}\frac{(x_1-x_3)x_4}{2}.
  \label{eq:a1-local-features}
\end{equation}
All logits are in $[-L,L]$.  These are genuinely incompatible conditional
families: at a context where the local feature vanishes, the cross
difference of the target-0 logit in the reciprocal pair is
$2a+2\delta Z$, whereas the reverse target logit gives $2a$.

The diagnostic key is target 0 at context $(1,1,1,1)$.  Its mean logit is
$\mu_Z=a+\delta Z$ and its local feature is zero.  Each charged query
returns the teacher logit plus a fresh independent noise
$\xi\in\{-\nu,+\nu\}$ with equal probability.  Thus the 18 repeated
diagnostic calls are 18 noisy observations, unlike repeated keys in the
exact-table protocol.  The common selector chooses the nearer of
$\mu_{-1}$ and $\mu_{+1}$ to their sample mean, with a fixed lexical
tie rule.

For each regime, the two registered local keys are target-0 contexts
\begin{align*}
  D_Z^A&=\{(Z,1,-1,1),(Z,-1,1,1)\},\\
  D_Z^B&=\{(1,Z,-1,1),(-1,Z,1,1)\}.
\end{align*}
Their local features in the correct regime are $+1$ and $-1$.
The adaptive policy chooses $D_{\widehat Z}^F$ after diagnosis.
Every deterministic static policy fixes an ordered pair from
$D_{-1}^F\cup D_{+1}^F$ before seeing responses, giving $4^2=16$
possibilities; response-independent random mixtures are also admitted.
Both kinds of policy spend the same 18 diagnostic and two local calls.

The common learner uses diagnostic calls only for $\widehat Z$.
Given two local responses $Y_i$ at keys $k_i$, it subtracts the public
known-logit map $h_{F,\widehat Z}(k_i)$ and fits the scalar coefficient
\begin{equation}
  \widehat\tau
  :=\operatorname*{arg\,min}_{t\in\mathbb R}
  \left\{\frac12\sum_{i=1}^2
  \big(Y_i-h_{F,\widehat Z}(k_i)
       -\phi_{\widehat Z}^F(k_i)t\big)^2+t^2\right\}.
  \label{eq:a1-common-ridge}
\end{equation}
The strictly convex objective has the unique minimizer
\[
  \widehat\tau
  =\frac{\sum_{i=1}^2\phi_{\widehat Z}^F(k_i)
                 (Y_i-h_{F,\widehat Z}(k_i))}
         {2+\sum_{i=1}^2(\phi_{\widehat Z}^F(k_i))^2}.
\]
Here $h_{F,z}$ is the corresponding $a x_1$ or $a x_2$ pairwise term,
$\delta z\psi$ at target 0, the reciprocal $a x_0$ term at target 1 or 2,
and zero elsewhere; $\phi_z^F$ is zero off target 0.
The reported full-key logit is
$\operatorname{clip}_{[-L,L]}
(h_{F,\widehat Z}+\phi_{\widehat Z}^F\widehat\tau)$.
This output is a conditional table, not a coherent joint.  Indeed, in family
$A$ fix $x_2=x_3=x_4=1$ and vary $x_1$.  The local feature vanishes and
clipping is inactive; the reported target-0 logit changes by
$2a+2\delta\widehat Z$, while the reciprocal target-1 logit changes by
$2a$ when $x_0$ varies.  A positive joint requires these cross-differences
to agree.  In family $B$ the same mismatch uses targets 0 and 2, fixing
$x_1=x_3=x_4=1$ and varying $x_2$.  Thus no joint can in general induce the
reported table, irrespective of the fitted $\widehat\tau$.
At budget 18, no local coefficient is fitted and the common default is
$\widehat\tau=0$, giving identical outputs for all policies.
\Needspace{20\baselineskip}
\section{Restricted response-adaptive separation}
\label{app:restricted-separation}

\begin{theorem}[Restricted conditional-table response-adaptive separation]
\label{thm:restricted-adaptivity}
Fix $F\in\{A,B\}$, the corresponding binary conditional-logit family in
Appendix~\ref{app:adaptive-families}, a uniform hidden regime
$Z\in\{-1,+1\}$, and $\tau\in[0.19,0.20]$.  Every policy uses the same
18-call diagnostic selector, two local calls with independent Rademacher
logit noise of magnitude $0.05$, the two-row ridge learner with $\lambda=1$,
and the auxiliary full-key conditional-table loss $\widetilde R_{\rm full}$
in Eq.~\eqref{eq:auxiliary-table-risk}.  Let $\Pi_{\rm stat}^F$ contain all
16 deterministic ordered allocations of the two local calls among the four
frozen local keys and all their response-independent randomizations.  At
budget 20, the response-adaptive allocation has lower expected risk than
every policy in $\Pi_{\rm stat}^F$; at budget 18, all policies have the same
diagnostic calls and learner and have equal risk.  Therefore, with equal E-AUC
weights on the two budgets,
\begin{equation}
  \inf_{\pi\in\Pi_{\rm stat}^F}
  \mathrm{E\text{-}AUC}_{F,\tau}(\pi)
  -\mathrm{E\text{-}AUC}_{F,\tau}(\pi_{\rm adapt})
  \ge 7\times10^{-6},
  \qquad F\in\{A,B\},
  \label{eq:restricted-adaptivity}
\end{equation}
\end{theorem}

The families and common learner are specified in
Appendix~\ref{app:adaptive-families}; the proof is in
Appendix~\ref{app:adaptive-proof}.  This is a finite-construction result
for the stated conditional-table task.  The Gaussian and public-data-weighted
studies in Appendix~\ref{app:supplementary-studies} examine distinct
observation laws and weighting, including the observed transfer failure.
\Needspace{9\baselineskip}
\section{Proof of the restricted separation}
\label{app:adaptive-proof}

\begin{proof}[Proof of Theorem~\ref{thm:restricted-adaptivity}]
The two diagnostic means are separated by $2\delta=0.1$.  A
sub-Gaussian tail bound for the average of 18 independent noises gives,
for each $Z$ and either family,
\begin{equation}
  \Pr(\widehat Z\ne Z\mid Z)
  \le\epsilon
  :=\exp\!\left(-\frac{18(2\delta)^2}{8\nu^2}\right)
  =\exp(-9).
  \label{eq:a1-diagnostic-error}
\end{equation}
The local noises are independent of the diagnostic transcript.
On correct diagnosis, let $n\in\{0,1,2\}$ be the number of local calls
whose feature is nonzero for the true regime.  From
Eq.~\eqref{eq:a1-common-ridge},
\[
  \widehat\tau_n
  =\frac{n\tau+\sum_{i=1}^{n}\phi_i\xi_i}{n+2},
  \qquad
  \mathcal E_\tau(n):=\mathbb E(\widehat\tau_n-\tau)^2
  =\frac{4\tau^2+n\nu^2}{(n+2)^2}.
\]
The diagnostic rows do not enter this two-row ridge normalization.
For every $\tau$ in the stated interval, $\mathcal E_\tau(n)$ is decreasing
and convex on $[0,2]$: direct differentiation gives a negative first
derivative because $\nu^2<4\tau^2$, and a nonnegative second derivative
because $\nu^2<3\tau^2$.

The local feature is nonzero at four of the 16 target-0 contexts and zero
for the other four targets.  Thus the full-key risk weight of the scalar
error on correct diagnosis is $w=4/(5\cdot16)=1/20$.
The true and reported logits lie in $[-L,L]$, and the clipping is inactive
on correct diagnosis.  Logistic curvature therefore bounds the risk there
between $(m_L/2)w\mathcal E_\tau(n)$ and
$(1/8)w\mathcal E_\tau(n)$, where
$m_L=\sigma(L)(1-\sigma(L))$.  On wrong diagnosis, clipping bounds the
entire full-key risk by $L^2/2$.

For any deterministic static ordered pair, let $n_+$ and $n_-$ count its
calls in the two regime blocks.  Since $n_++n_-\le2$ and $Z$ is uniform,
$\mathbb E n_Z\le1$.
This remains true under response-independent
randomization.  It would be incorrect to apply this count after conditioning
on correct diagnosis, because the selector's lexical tie rule can change
the conditional regime weights.  Instead, fix $Z=z$ and the static
allocation randomness $U$, and define the success probability
$s_z:=\Pr(\widehat Z=z\mid Z=z)\ge1-\epsilon$.
Independence of local noise gives
\[
  \mathbb E[\widetilde R_{\rm full}^{\rm stat}
             \mathbf 1_{\{\widehat Z=Z\}}
             \mid Z=z,U]
  \ge s_z\frac{m_L}{2}w\mathcal E_\tau(n_z).
\]
Now average over the unconditioned uniform $Z$ and independent $U$.
Jensen's inequality and monotonicity of $\mathcal E_\tau$ yield
\begin{equation}
  \mathbb E\widetilde R_{\rm full}^{\rm stat}
  \ge(1-\epsilon)\frac{m_L}{2}w\mathcal E_\tau(1).
  \label{eq:a1-static-lower}
\end{equation}
The adaptive policy has $n=2$ on correct diagnosis.  Its unconditional
risk therefore satisfies
\begin{equation}
  \mathbb E\widetilde R_{\rm full}^{\rm adapt}
  \le(1-\epsilon)\frac{w}{8}\mathcal E_\tau(2)
       +\epsilon\frac{L^2}{2}.
  \label{eq:a1-adaptive-upper}
\end{equation}
Here replacing the actual error probability by $\epsilon$ is valid
because the bad-event bound exceeds the correct-event bound.

Both bounds hold for each fixed family and every
$\tau\in[0.19,0.20]$.  The static lower expression is minimized at
$\tau=0.19$, while the active upper expression is maximized at
$\tau=0.20$.  Direct substitution gives
\[
  \mathcal E_{0.19}(1)=0.016322222\ldots,
  \quad \mathcal E_{0.20}(2)=0.0103125,
\]
\[
  \operatorname{LB}_{20}^{\rm stat}>9.5882\times10^{-5},
  \qquad
  \operatorname{UB}_{20}^{\rm adapt}<7.9872\times10^{-5}.
\]
Their difference exceeds $1.60\times10^{-5}$.
At budget 18 the shared diagnostic transcript, selector, and default
coefficient yield exactly equal risks.  Half-weighting budgets 18 and 20
therefore gives an E-AUC difference greater than
$8.0\times10^{-6}>7\times10^{-6}$ for each of the 16 static policies
and any response-independent mixture.
\end{proof}
\section{Supplementary bounds and query studies}
\label{app:supplementary-studies}

The fixed-pool study measures quantized parameter recovery; the Gaussian
and semisynthetic studies assess the non-joint conditional-table task in
Appendix~\ref{app:auxiliary-query-design}.  Their protocols and estimands are
compared in Table~\ref{tab:supplementary-claim-evidence-bridge}.

\subsection{Fixed-design upper bound and approximation}
\label{sec:structured-upper}

Use the pairwise joint and feature rows of
Eqs.~\eqref{eq:pairwise-student}--\eqref{eq:pairwise-logit}.
In this subsection, let $\mu$ be any fixed,
sampleable distribution over such keys, let the oracle return a deterministic
exact probability $p(z)\in[0,1]$, and draw $B$ keys independently from $\mu$.
If $\widehat\theta_B$ minimizes empirical conditional-KL risk over
$\Theta_R$, write $M=R\sqrt d$ and $C_M=\log(1+e^M)$.  For every
$\delta\in(0,1)$, the structured oracle inequality is
\begin{equation}
\Pr\!\left\{
  R_\mu(q_{\widehat\theta_B};\mathcal O)
  -\inf_{\theta\in\Theta_R}R_\mu(q_\theta;\mathcal O)
  \le 4R\sqrt{\frac dB}
  +2C_M\sqrt{\frac{\log(2/\delta)}{2B}}
\right\}\ge1-\delta.
\label{eq:structured-upper}
\end{equation}
This applies standard Rademacher uniform convergence
\cite{bartlett2002rademacher}; Appendix~\ref{app:structured-upper-proof}
checks the loss envelope, feature norm, and constants.  The comparator
remains the best coherent pairwise projection for the given table.

If, additionally, the oracle logit is $a(z)^\top\theta_0+r(z)$ for some
$\theta_0\in\Theta_R$ and $\mathbb E_\mu r(Z)^2\le\rho^2$, then the
logistic log-partition's $1/4$ smoothness gives
\begin{equation}
  I_\mu(\mathcal O)+A_{\mu,\mathcal Q_R}(\mathcal O)
  =\inf_{\theta\in\Theta_R}R_\mu(q_\theta;\mathcal O)
  \le\rho^2/8.
  \label{eq:bounded-residual}
\end{equation}
This controls the sum of irreducible incompatibility and structural
approximation; it does not identify either term separately.

\subsection{Budget trend in a frozen structured family}
\label{sec:structured-experiment}

The structured study uses a frozen family and a fixed pool of candidate
queries.  Checks of held-out joint risk, support, and parameter reporting
covered 450 eligible rows across dimensions $16$, $32$, and $64$.
The independently frozen quantized study uses resolution $\eta=10^{-3}$.  For each target, the
complete frozen candidate pool $\mathcal C$ contains 4,096 keys, its design
matrix is $A_{\mathcal C}\in\mathbb R^{4096\times d}$, and
$\ell_{\eta,\mathcal C}\in\mathbb R^{4096}$ is the vector of quantized
oracle logits.  The evaluation reference is the unique full-column-rank
least-squares target
\begin{equation}
  \beta_{\eta,\mathcal C}
  :=\operatorname*{arg\,min}_{\beta\in\mathbb R^d}
  \|A_{\mathcal C}\beta-
  \ell_{\eta,\mathcal C}\|_2^2.
  \label{eq:p2-fixed-pool-target}
\end{equation}
The stored design audit verifies that $A_{\mathcal C}$ has full column rank,
so the displayed minimizer is single-valued.  The learner instead uses the
first $B\in\{128,256,512\}$ rows of the fixed ordering and reports the
corresponding prefix least-squares estimate $\widehat\beta_B$.  The 4,096
reference calls per target serve to
define the evaluation target and are accounted for separately from the
learner's $B$ calls.  Table~\ref{tab:p2-budget} reports the median parameter
distance $\|\widehat\beta_B-\beta_{\eta,\mathcal C}\|_2$ among eligible
targets at each budget.  The median support F1 and the
support-certificate-stable fraction were both $1$ in all nine
dimension--budget cells.

\begin{table}[H]
  \centering
  \small
  \caption{Structured quantized recovery against the frozen fixed-pool target.
  Entries at each budget are median quantized-prefix excess $\ell_2$ error
  in units of $10^{-4}$.  Slopes and $R^2$ use only the three displayed
  budget points for each dimension; the last column is the fit $R^2$.}
  \label{tab:p2-budget}
  \begin{tabular}{rrrrrr}
    \toprule
    $d$ & $B=128$ & $B=256$ & $B=512$ & log--log slope & fit $R^2$ \\
    \midrule
    16 & 1.93670 & 1.17407 & 0.853482 & $-0.591$ & 0.984 \\
    32 & 3.03878 & 2.00064 & 1.26002  & $-0.635$ & 0.999 \\
    64 & 4.93385 & 2.85877 & 1.80692  & $-0.725$ & 0.998 \\
    \bottomrule
  \end{tabular}
\end{table}

The log--log slopes in Table~\ref{tab:p2-budget} fit exactly the three
registered budget points for each dimension.  Their negative signs show a
within-family decrease in this particular fixed-pool error measure.  The
earlier development track did not meet its frozen all-dimension stability
criterion; the independently frozen study reported here passed it.  Neither
the slopes nor
the support scores establish a general finite-bit or $\varepsilon$-dependent
rate, an active-design advantage, or the expected conditional-KL minimax
rate of the exact-table theorem.

\subsection{Restricted Gaussian query-design comparison}
\label{sec:gaussian-experiment}

The Gaussian comparison keeps the two frozen conditional-logit families, two
parameter endpoints, common learner, and all 16 response-independent
two-local-call static policies from the restricted separation, but uses
fresh Gaussian logit noise with standard deviation $0.05$ rather than the
theorem's bounded Rademacher noise.  Its registered
core consists of eight anchors and 60 replications per anchor.  Every one of
the $3{,}840$ recorded paired static-minus-active E-AUC gaps is positive, as
is the worst-policy score in each of the 60 replication indices.  The mean
paired gap is $3.305842798\times10^{-5}$ and the smallest recorded gap is
$5.560061386\times10^{-7}$.  The pooled active mean E-AUC is about $18.1\%$
below the pooled static mean under the study's descriptive denominator.
These observations describe conditional-table prediction in the frozen
Gaussian construction.  They establish neither a coherent-joint active
advantage nor an expanded comparator class or observation law.

\subsection{Public-data structural weighting and transfer boundary}
\label{sec:semisynthetic-experiment}

The semisynthetic bridge uses five public OpenML datasets: letter
(ID 6)~\cite{dataset_letter}, mfeat-morphological
(ID 18)~\cite{dataset_mfeat}, mushroom (ID 24)~\cite{dataset_mushroom},
nursery (ID 26)~\cite{dataset_nursery}, and page-blocks
(ID 30)~\cite{dataset_page_blocks}, all version 1.  Three fixed splits and
four generator families produce 60 reference units.  The public
data determine the selected binary context structure and weights; the
conditional-logit oracle in this comparison is generated.  This is a
real-structure stress test of the restricted auxiliary conditional-table
result, not a five-dataset evaluation of a real tabular foundation model or
of active coherent reconstruction.

For each reference unit and replication, the recorded contrast is the
minimum static-minus-adaptive E-AUC gap over the frozen family/parameter
anchors and 16 static policies.  A reference score averages its 60
replication minima; a dataset score averages its 12 reference scores (three
splits and four generators).  Thus the plotted statistic is a nested
worst-comparator contrast, not an average-static-policy gap.

The registered global score, all four generator scores, and 53 of 60
reference scores favor the response-adaptive policy.  The stronger registered
all-dataset requirement fails: only four of five dataset scores favor it, and
mushroom is negative.  In mushroom split 0, all four generator references are
negative and 234 of 240 replication minima are nonpositive.  The registered
decision remains \texttt{bridge\_not\_passed}; positive aggregate results do
not reverse that decision.  Table~\ref{tab:transfer-boundary} displays the
positive core evidence beside the observed transfer failure, while
Figure~\ref{fig:transfer-dataset-scores} shows the five dataset-level signs.

\Needspace{7\baselineskip}
A post-hoc within-design accounting attributes $98.42\%$ of the mushroom
split-0 negative contribution mass to three high-weight contexts.  This
locates the arithmetic concentration of the reversal; it does not identify a
causal feature or establish how to prevent the failure.  The permitted
conclusion is an aggregate and majority-reference advantage with a verified
dataset/split heterogeneity boundary.

\begin{table}[t]
  \centering
  \small
  \caption{Restricted query-design evidence and its transfer boundary.
  Positive auxiliary conditional-table contrasts favor the
  response-adaptive policy over the frozen response-independent static
  comparator.  The semisynthetic all-dataset gate was not met.}
  \label{tab:transfer-boundary}
  \begin{tabular}{@{}>{\raggedright\arraybackslash}p{0.22\linewidth}>{\raggedright\arraybackslash}p{0.34\linewidth}>{\raggedright\arraybackslash}p{0.34\linewidth}@{}}
    \toprule
    Study & Positive evidence & Registered boundary \\
    \midrule
    Gaussian core & $3{,}840/3{,}840$ paired gaps positive;
    mean gap $3.3058\times10^{-5}$ & Two families, two endpoints,
    16 static policies \\
    Semisynthetic bridge & Global and $4/4$ generator scores positive;
    $53/60$ references positive & $4/5$ dataset scores positive;
    mushroom negative; all-dataset gate failed \\
    \bottomrule
  \end{tabular}
\end{table}
\begin{figure}[!htbp]
  \centering
  \includegraphics[width=0.84\linewidth]{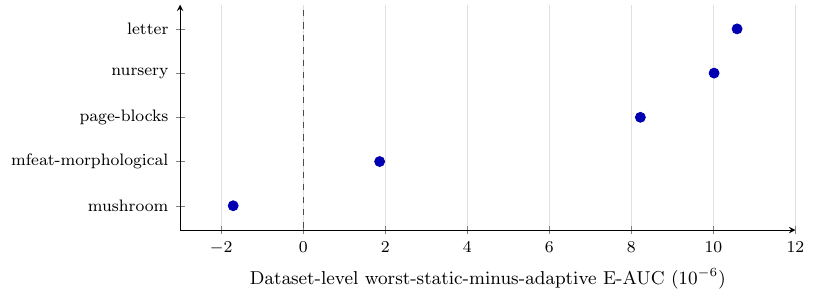}
  \caption{Registered semisynthetic boundary for the auxiliary
  conditional-table task. Each reference score averages 60 replication-wise
  minima of static-minus-adaptive E-AUC gaps (the two-budget expected table-loss
  summary in Eq.~\eqref{eq:auxiliary-eauc}) over the frozen anchors and 16
  static policies; each point averages 12 reference scores for one public
  dataset (three splits, four generators). Positive values favor adaptation.
  The oracle is generated; the all-dataset gate failed because mushroom is
  negative. No coherent-joint or dataset-uniform active advantage is claimed.}
  \label{fig:transfer-dataset-scores}
\end{figure}

\subsection{Broader prior-work comparison}

Table~\ref{tab:broader-prior-work} retains additional related results.
Each row describes the cited result only; an absence of the present
finite-query statement in a row is not a literature-wide novelty claim.

\begingroup
\scriptsize
\setlength{\tabcolsep}{4pt}
\setlength{\LTcapwidth}{\linewidth}
\begin{longtable}{@{}>{\raggedright\arraybackslash}p{0.16\linewidth}>{\raggedright\arraybackslash}p{0.26\linewidth}>{\raggedright\arraybackslash}p{0.26\linewidth}>{\raggedright\arraybackslash}p{0.23\linewidth}@{}}
\caption{Assumption and target comparison with closest prior work. ``Not analyzed''
means not part of the cited result, not an absence claim about the literature.}
\label{tab:broader-prior-work}\\
\toprule
Work & Input/oracle assumption & Output and objective & Finite-query result \\
\midrule
\endfirsthead
\multicolumn{4}{@{}l}{\tablename~\thetable\ (continued)}\\
\toprule
Work & Input/oracle assumption & Output and objective & Finite-query result \\
\midrule
\endhead
\midrule
\multicolumn{4}{r@{}}{Continued on next page}\\
\endfoot
\bottomrule
\endlastfoot
Majid et al.~\cite{majid2025consistent} & Family of subset conditionals; joint compatibility at issue & Conditions for consistent joint recovery & Not analyzed \\
Arnold et al.~\cite{arnold2002exact} & Two finite discrete conditional families & Exact, near, and $\epsilon$-compatibility; compatible-table adjustment & Not analyzed \\
Mur\'e~\cite{mure2019optimal} & Potentially incompatible conditionals & Pseudo-Gibbs stationary compromise; mean-squared misfit & Not analyzed \\
Chen et al.~\cite{chen2011gibbs} & Nearly compatible or incompatible conditional models & Gibbs-ensemble joint compromise & Not analyzed \\
Heckerman et al.~\cite{heckerman2000dependency} & Locally specified dependency network; possible inconsistency & Gibbs-like inference and joint behavior & Not for this oracle \\
Picado Mui\~no~\cite{picadomuino2011repairing} & Explicit probabilistic knowledge base & Consistency-restoring probability adjustments & Not analyzed \\
Thimm~\cite{thimm2013inconsistency} & Explicit probabilistic conditional logic & Distance-based inconsistency measures & Not analyzed \\
Potyka--Thimm~\cite{potyka2014consolidation} & Explicit probabilistic conditional knowledge base & Consolidation through generalized models & Not analyzed \\
Nguyen et al.~\cite{nguyen2021probabilistic} & Explicit probabilistic knowledge bases, possibly inconsistent and structurally diverse & Divergence-based consistency recovery and merging & Not analyzed \\
Barik--Honorio~\cite{barik2019learning} & Conditional queries; low-rank BN tables and compatible teacher & Exact directed BN structure & Polynomial under stated assumptions \\
Torroba Hennigen--Kim~\cite{torroba2023deriving} & Masked-LM conditionals; two-token analysis & Explicit joint; conditional closeness & Not analyzed \\
Liu et al.~\cite{liu2023understanding} & Deep generative teacher & Tractable probabilistic circuit distillation & Not for this conditional oracle \\
Blanca et al.~\cite{blanca2023complexity} & Coordinate-conditional samples from a hidden joint & Identity test & Bounds for their sampling oracle \\
Kl\"otergens et al.~\cite{klotergens2026agree} & TFM predictive conditionals & Consistency diagnostics & Not a repair guarantee \\
This work & Possibly incompatible, probability-vector conditional oracle & Coherent-joint projection; conditional-KL excess risk & Matching minimax only for the frozen passive-i.i.d. exact-table protocol and window \\
\end{longtable}
\endgroup

\end{document}